\documentclass[11pt]{amsart}

\usepackage[margin=1.05in]{geometry}
\usepackage{amsmath,amssymb,amsthm,mathtools}
\mathtoolsset{showonlyrefs=true}
\usepackage{enumitem}
\usepackage{microtype}
\usepackage{hyperref}
\hypersetup{urlcolor=blue, citecolor=blue, linkcolor=blue, colorlinks=true}
\usepackage{tikz-cd}
\usepackage{tcolorbox}
\usepackage{csquotes}
\usepackage[nameinlink,noabbrev, capitalise]{cleveref}
\usepackage[style=alphabetic,maxnames=5,maxalphanames=5,doi=true,backend=bibtex,isbn=false,url=false]{biblatex}
\bibliography{ref.bib}
\renewbibmacro{in:}{}
\renewbibmacro*{issue+date}{%
  \ifboolexpr{not test {\iffieldundef{year}} or not test {\iffieldundef{issue}}}
    {\printtext[parens]{%
       \iffieldundef{issue}
         {\usebibmacro{date}}
         {\printfield{issue}%
          \setunit*{\addspace}%
          \usebibmacro{date}}}}
    {}%
  \newunit}

\usepackage{todonotes}

\numberwithin{equation}{section}

\newtheorem{theorem}{Theorem}[section]
\newtheorem{proposition}[theorem]{Proposition}
\newtheorem{lemma}[theorem]{Lemma}
\newtheorem{corollary}[theorem]{Corollary}

\newtheorem{remark}[theorem]{Remark}

\newtheorem{bigthm}{Theorem}
\newtheorem{bigcor}[bigthm]{Corollary}

\usepackage[nameinlink,noabbrev, capitalise]{cleveref}

\AddToHook{env/proposition/begin}{\crefalias{theorem}{proposition}}
\AddToHook{env/lemma/begin}{\crefalias{theorem}{lemma}}
\AddToHook{env/corollary/begin}{\crefalias{theorem}{corollary}}
\AddToHook{env/remark/begin}{\crefalias{theorem}{remark}}
\AddToHook{env/definition/begin}{\crefalias{theorem}{definition}}

\newcommand{\dist}{\operatorname{dist}}

\newcommand{\dd}{\,\mathrm d}

\newcommand{\R}{\mathbb R}
\newcommand{\Riem}{\mathcal{R}}
\newcommand{\calQ}{\mathcal{Q}}

\newcommand{\scpr}[1]{\left\langle#1\right\rangle}

\newcommand{\scal}{\mathrm{R}}
\newcommand{\Laplace}{\Delta}
\newcommand{\embeds}{\hookrightarrow}
\newcommand{\dvol}{\mathrm{dvol}}
\newcommand{\pt}{{\mathrm{pt}}}

\DeclareMathOperator{\Crit}{Crit}

\DeclareMathOperator{\codim}{codim}
\newcommand{\Q}{\mathcal Q}

\let\ge\geqslant
\let\geq\geqslant
\let\le\leqslant
\let\leq\leqslant

\title{Spaces of metrics with positive spectral scalar curvature}
\author{Gioacchino Antonelli}
\address[Gioacchino Antonelli]{Department of Mathematics, University of Notre Dame, Hurley Hall, 255 Hurley, Notre Dame, IN 46556, United States}
\email{gantonel@nd.edu}
\author{Georg Frenck}
\address[Georg Frenck]{Institut f\"ur Mathematik,
Universit\"at Augsburg, Universit\"atsstra{\ss}e 14,
86159 Augsburg, Germany}
\email{georg.frenck@math.uni-augsburg.de}
\author{Bernhard Hanke}
\address[Bernhard Hanke]{Institut f\"ur Mathematik,
Universit\"at Augsburg, Universit\"atsstra{\ss}e 14,
86159 Augsburg, Germany}
\email{hanke@math.uni-augsburg.de}
\date{\today}

\begin{document}
 
\begin{abstract}
Let $n\geq2$ and let $M^n$
be a closed connected smooth manifold. 
Let $\mathcal{R}^\gamma(M)$ be the space of smooth
Riemannian metrics $g$ on $M$ for which
the generalized conformal Laplace operator $-\gamma\Delta_g+\scal_g$ is strictly positive.

We prove that if $n=2$ and $\gamma>0$, or if $n\ge3$ and $0< \gamma \leq 4(n-1)/(n-2)$, the inclusion $\mathcal{R}^0(M)\hookrightarrow \mathcal{R}^\gamma(M)$ is a homotopy equivalence, thus generalizing, to all dimensions and in the maximal range, the results of Botvinnik--Rosenberg and Li--Mantoulidis.
Then, we prove that if $n\ge3$ and $\gamma>4(n-1)/(n-2)$, the space $\mathcal{R}^\gamma(M)$ is contractible, and hence nonempty.
This solves a homotopy-theoretic strengthening of a conjecture of Gromov \cite[Conjecture 3, Section 6.1.2]{GromovFourLectures} in the maximal possible coefficient range. 

Concerning Gromov's conjecture we also treat the equivariant case and the case of manifolds with boundary.
\end{abstract}
\maketitle

\tableofcontents 

\section{Introduction}

\noindent In this paper, all manifolds $M$ considered will be smooth and connected (possibly nonorientable). 
The term closed manifold stands for compact manifold without boundary. 
For $\gamma\in[0,\infty)$ and a compact manifold $M$ of dimension $n$, possibly with nonempty boundary, consider the space
\begin{equation}\label{eq:def-r-gamma}
  \Riem^{\gamma}(M)\coloneqq\{g\in\Riem(M)\mid  -\gamma\Delta_g + \scal_{g}>0\},
\end{equation}
where $\Riem(M)$ is the space of smooth Riemannian metrics on $M$ equipped with the $C^\infty$-topology, $\Delta_g$ is the (non-positive) Laplace operator and $\scal_{g}$ is the scalar curvature of the metric $g$. 
Throughout the paper, the notation $-\gamma\Delta_g+\scal_g>0$ means
\begin{equation}\label{eqn:-gammaDelta+r}
-\gamma\Delta_g+\scal_{g}>0\  \stackrel{\mathrm{def}}{\Longleftrightarrow}  \lambda_1^N(-\gamma\Delta_g+\scal_g)\coloneqq \inf_{0\not\equiv f\in C^\infty(M)}\frac{\int_M \left(\gamma|\nabla_g f|_g^2 + \scal_g f^2\right)\dvol_g}{\int_M f^2\dvol_g} >0.
\end{equation}
In the following, $\dvol_g$ denotes the $n$-dimensional Hausdorff measure. Note that when $\partial M\neq\varnothing$, $\lambda_1^N(L)$ denotes the first Neumann eigenvalue of $L$. From \eqref{eqn:-gammaDelta+r}, there are inclusions
\begin{equation}\label{eq:inclusion-of-R-gamma}
\Riem^{\gamma}(M)\embeds\Riem^{\gamma'}(M) \quad \textrm{ for } \gamma < \gamma' .
\end{equation}
Note that $\Riem^0(M)$ is the space of positive-scalar-curvature (PSC) metrics.

The spaces $\mathcal{R}^\gamma(M)$ appear naturally in geometric analysis. When $M$ is closed, for $n \geq 3$ and  $\gamma = \frac{4(n-1)}{n-2}$, the operator $- \gamma \Delta_g + \scal_g$ is the classical conformal Laplace operator, and the space $\Riem^{\gamma}(M)$ consists of the metrics whose conformal classes have positive Yamabe constant \cite{KazdanWarner}. At $\gamma=4$, the first eigenvalue of $-4\Delta_g+\scal_g$ is the Perelman
$\lambda$-functional \cite{PerelmanEntropy}. At $\gamma=2$, Li--Mantoulidis \cite[Lemma C.6]{LiMantoulidis2023}
identify $\mathcal R^2(M)$ with the space of metrics realizable as
induced metrics on two-sided stable minimal hypersurfaces in PSC
manifolds. Geometric inequalities for metrics $g\in \mathcal{R}^\gamma(M)$ have been studied, e.g., in \cite{HirschKazarasKhuriZhang}.

The following is our main result. In this paper, for all $n\geq 3$ let $\gamma_n:=\frac{4(n-1)}{n-2}$.
\begin{bigthm}\label{main:contractibility}
  Let $n\ge2$ and let $M$ be a closed $n$-dimensional manifold. 
  \begin{enumerate}
    \item[\rm(1)] If $n=2$ and $\gamma> 0$, or if $n\ge3$ and $0 <  \gamma\le\gamma_n$, the inclusion \(\Riem^0(M)\embeds\Riem^\gamma(M)\) is a homotopy equivalence.
    \item[\rm(2)] If $n\ge3$ and $\gamma>\gamma_n$, the space $\Riem^\gamma(M)$ is contractible.
  \end{enumerate}
\end{bigthm}

\noindent We list some immediate consequences of \cref{main:contractibility}.

\begin{bigcor}\leavevmode
  \begin{enumerate}
    \item [a)]
    For $\gamma\ge0$ and a closed manifold $M$ of dimension at least $3$, the space $\Riem^{\gamma}(M)$ is nonempty if and only if either $\gamma\le \gamma_n$ and $M$ admits a PSC metric, or $\gamma>\gamma_n$.
    \item [b)]
    Let $n\geq 3$. \cref{main:contractibility}(2) confirms a conjecture of Gromov \cite[Conjecture 3, Section 6.1.2]{GromovFourLectures}, which predicts the existence of a constant $c_n\ge\gamma_n$ such that for all $\gamma>c_n$ and every closed $n$-manifold $M^n$, there is a smooth metric $g$ on $M$ such that $-\gamma\Delta_g+\scal_g\geq 0$. In fact \cref{main:contractibility}(2) shows that $c_n=\gamma_n$ works.
    \item [c)]
    Let $n\geq 3$. Let $0\leq\gamma <\gamma'$, and assume $\gamma'\leq\gamma_n$ or $\gamma>\gamma_n$. Then the inclusion
$\Riem^\gamma(M)\hookrightarrow\Riem^{\gamma'}(M)$
    is a homotopy equivalence.
  \end{enumerate}
\end{bigcor}

The following remark relates \cref{main:contractibility} to some previous results.

\begin{remark}\leavevmode
\begin{enumerate}
\item[a)] For a closed manifold $M$ of dimension at least $3$ admitting a PSC metric, \cref{main:contractibility}(1) for $\gamma = \gamma_n$ is proved by Botvinnik--Rosenberg \cite[Theorem 3.4]{BotRos}. See also Akutagawa--Botvinnik \cite[Theorem 7.1]{AkutagawaBotvinnik-relative-yamabe-invariant}.
\item [b)] For a closed orientable  $3$-manifold $M$ admitting a PSC metric,
    Li--Mantoulidis \cite[Theorems 1.3--1.4]{LiMantoulidis2023}\footnote{
        Li--Mantoulidis consider the space $\{g \mid \lambda_1(-\Delta_g+k\scal_g)>0\}$, which in our notation corresponds to $\Riem^{1/k}(M)$.
      }prove that the space $\Riem^\gamma(M)$ is path-connected for $0\le\gamma\le4$ and contractible for $\gamma=8$.
    Together with \cite[Theorem 1.1]{BamlerKleiner2019}, \cref{main:contractibility} strengthens Li--Mantoulidis' result showing that $\Riem^\gamma(M)$ is contractible for all $\gamma \ge 0$. 
    
    Moreover, \cref{main:contractibility} answers the question raised in
\cite[p. 1833]{LiMantoulidis2023} of understanding the topology of the spaces $\mathcal{R}^\gamma(M)$. For previously related results when $M=\mathbb S^2$, see \cite{MantoulidisSchoen2015}.
\item [c)] For a closed manifold $M$  of dimension at least $4$ admitting a PSC metric, it is well-known that the homotopy type of the space of PSC metrics, $\Riem^{\gamma = 0}(M)$, is often highly nontrivial; see, for example,  \cite{BERW, CS, ERW, HSS14, Rub}. 
As we increase $\gamma$, according to \cref{main:contractibility}, the homotopy type of $\Riem^{\gamma}(M)$ remains unchanged until $\gamma$ reaches the threshold of $\gamma_n$, after which a ``phase transition'' happens and it abruptly collapses for larger values of $\gamma$.
\item[d)] It makes sense to define $\Riem^{\gamma = \infty}(M)$ as the space of metrics $g$ for which $- \Delta_g$ is strictly positive. 
Then we have $\Riem^{\infty}(M) = \varnothing$.
\end{enumerate}
\end{remark}   
\smallskip 

Concerning the conjecture of Gromov, we also have the following existence result for manifolds with boundary and  metrics that are invariant under group actions.  
Let us recall our sign convention for the second fundamental form: 
if $\nu$ is the unit outward normal to $\partial M$, and $X,Y\in T\partial M$, then 
\[
\mathrm{II}_g(X,Y)
  \coloneqq g(\nabla^g_X\nu,Y).
\]

\begin{bigthm}\label{main:existence-with-invariance-and-boundary}
    Let $n\geq 3$, and let $\gamma>\gamma_n$. Let $\Gamma$ be a compact Lie group,  and let $M^n$ be a compact $n$-dimensional $\Gamma$-manifold, possibly with nonempty boundary. 
    If $\partial M=\varnothing$, assume in addition that the $\Gamma$-action is not transitive. If $\partial M\neq \varnothing$, let $q$ be a $\Gamma$-invariant Riemannian metric on $\partial M$, and let
$k$ be a $\Gamma$-invariant symmetric $(0,2)$-tensor on $\partial M$.

    Then there exists a $\Gamma$-invariant metric $g$ on $M$ that satisfies
    \[
    -\gamma\Delta_g+\scal_g>0, \qquad {\rm II}_g =k, \qquad g|_{T\partial M}=q.
    \]
    (The boundary conditions are omitted when $\partial M=\varnothing$.)
\end{bigthm}

\begin{remark} If $M^n$ is a torus $T^n$ and $\Gamma=T^n$ all $\Gamma$-invariant metrics are flat, so we cannot have $-\gamma\Delta+\scal>0$.
This shows that we cannot drop the assumption that the $\Gamma$-action is not transitive in the above theorem.
\end{remark}

\subsection{Strategy of the proof}
The proof of \cref{main:contractibility} reflects two distinct conformal mechanisms. In the proof of \cref{main:contractibility}(1), when $\gamma>0$, we introduce the space of positive supersolutions
\[
Z := \bigl\{(g,u)\in \mathcal R^\gamma(M)\times C^\infty(M,(0,\infty))\mid (-\gamma\Delta_g+\scal_g)u>0\bigr\}.
\]
On the one hand, we notice that the projection $Z\to\mathcal R^\gamma(M)$ is a homotopy equivalence, with homotopy inverse $g\mapsto (g,u_g)$, where $u_g$ is the positive first eigenfunction of $-\gamma\Delta_g+\scal_g$ with unit $L^2$-norm. On the other hand, we show that the conformal map
\[
(g,u)\longmapsto u^{\gamma/(n-1)}g
\]
is homotopy inverse to the map $\mathcal R^0(M)\to Z$, $g\mapsto(g,1)$. For this last computation, it is crucial that $\gamma\leq\gamma_n$ when $n\geq 3$. From the previous two properties, the proof of \cref{main:contractibility}(1) readily follows. See \cref{sec:A1} for the detailed proof following this outline.
\smallskip

Moving on, one of the crucial points of the proof of \cref{main:contractibility}(2) is the following identity. 
Writing $g=e^{2\varphi}h$ and $a:=2(n-1)/\gamma$, we have
\begin{equation}\label{eq:key-identity-intro}
\bigl(-\gamma\Delta_g+\scal_g\bigr)(e^{-a\varphi})
=
e^{-(2+a)\varphi}
\bigl(\scal_h+c_{n,\gamma}|\nabla_h\varphi|_h^2\bigr),
\end{equation}
where $
c_{n,\gamma}:=(n-1)(n-2)-\frac{4(n-1)^2}{\gamma}$ is positive because $\gamma>\gamma_n$; see \eqref{eqn:KeyRgammaMtoQcM}.

The latter identity \eqref{eq:key-identity-intro}, together with Barta's trick (see \cref{cor:Barta}), allows us to construct the continuous map $\mathcal Q_{c_{n,\gamma}}(M)\to \mathcal R^\gamma(M)$, $(h,\varphi)\mapsto e^{2\varphi}h$; here, for every $c>0$ we set
\[
\mathcal Q_c(M):=
\bigl\{(h,\varphi)\in\mathcal R(M)\times C^\infty(M)\mid
\scal_h+c|\nabla_h\varphi|_h^2>0\bigr\}.
\]
Conversely, using the positive first eigenfunction of $-\gamma\Delta_g+\scal_g$ with unit $L^2$-norm, we produce a continuous map $\mathcal{R}^\gamma(M)\to \mathcal{Q}_{c_{n,\gamma}}(M)$ in the opposite direction whose composition back to $\mathcal R^\gamma(M)$ is the identity; see \cref{sec:loctoglob}. Hence, to prove that $\mathcal{R}^\gamma(M)$ is contractible, it suffices to prove that $\mathcal Q_{c_{n,\gamma}}(M)$ is contractible. In this way, we translated the global condition that defines $\mathcal{R}^\gamma(M)$ to the local, thus more tractable, condition that defines $\mathcal Q_{c_{n,\gamma}}(M)$.
\smallskip

The fact that $\mathcal{Q}_c(M)$ is contractible when $c>0$ is the core and most technically demanding part of the proof of \cref{main:contractibility}(2). The proof is given in \cref{sec:qc-contractible}. The essential  input is a foliated genericity theorem of Bertelson \cite[Proposition~3.21]{Bertelson2002}, see \cref{thm:bertelson-isolated}, whose proof rests on Thom's jet transversality theorem, Mather's finite determinacy theorem, and classical results of Whitney on real algebraic sets. 

The proof that $\mathcal{Q}_c(M)$ is contractible when $c>0$ goes roughly as follows. We want to show that for every $k\geq 0$, every continuous map $(g,\varphi):\mathbb S^{k-1}\to\mathcal Q_c(M)$ extends to a continuous map $\mathbb D^k\to \mathcal Q_c(M)$. 
First, we prove that $\mathcal Q_c(M)\neq\varnothing$, thus addressing the case $k=0$. 
Hence, let us assume from now on $k>0$. 
After preliminary smoothing, one can regard this map as a compact (smooth) family $(g_\xi,\varphi_\xi)_{\xi\in\mathbb S^{k-1}}$, and interpolate it (in $\mathcal{R}(M)\times C^\infty(M)$) to a fixed base point in $\mathcal{Q}_c(M)$. 

The resulting interpolation $(g_{\xi,t},\varphi_{\xi,t})_{(\xi,t)\in\mathbb S^{k-1}\times[0,1]}$ need not remain in $\mathcal Q_c(M)$, so we have to correct it. 
Bertelson’s theorem allows us to perturb the function component $\varphi_{\xi,t}\leadsto \widehat \varphi_{\xi,t}$ so that the critical points of each $\widehat\varphi_{\xi,t}$ are isolated and therefore finite. 
Consequently, a parameter-dependent conformal deformation $g_{\xi,t}\leadsto \widehat g_{\xi,t}$ (see \cref{lem:local-improvement} after \cite{BaerHankeLocalFlexibility}) makes the scalar curvature positive on an open neighborhood $\Omega$ of the (compact) locus $C$ of the critical points
\[
 C=\{((\xi,t),m)\in \mathbb S^{k-1}\times[0,1]\times M\mid\mathrm{d}\widehat\varphi_{\xi,t}(m)=0\}.
\]
On the complement of $\Omega$, $|\nabla_{\widehat g_{\xi,t}}\widehat\varphi_{\xi,t}|$ is uniformly bounded below, so a sufficiently large rescaling of $\widehat\varphi_{\xi,t}$ makes $c|\nabla_{\widehat g_{\xi,t}}\widehat\varphi_{\xi,t}|^2$ dominate the negative part of the scalar curvature. 

Finally, if the perturbation $\widehat\varphi_{\xi,t}$ is chosen sufficiently close to $\varphi_{\xi,t}$, any resulting discrepancy in the endpoint data can be repaired by straight-line homotopies inside $\mathcal Q_c(M)$; see \cref{sec:qc-contractible} for a detailed proof. 
This shows that for every $k> 0$, every continuous map $\mathbb S^{k-1}\to\mathcal Q_c(M)$ extends to a continuous map $\mathbb D^k\to\mathcal{Q}_c(M)$, and thus $\mathcal Q_c(M)$ is weakly contractible. 
Since it is an open metrizable Fréchet manifold, a  theorem of Palais implies that it is contractible, concluding the proof.
\smallskip

The proof of \cref{main:existence-with-invariance-and-boundary} crucially exploits the key identity \eqref{eq:key-identity-intro} above, and the results of \cite{BaerHankeLocalFlexibility} and \cite{BaerHankeBoundaryConditions} to construct a $\Gamma$-invariant metric $h$ and a $\Gamma$-invariant function $\varphi$ such that $\scal_h+c_{n,\gamma}|\nabla_h\varphi|_h^2>0$ (with fixed boundary data if $\partial M\neq\varnothing$); compare with \cref{thm:mainnew2} and \cref{thm:mainnew}. We refer the reader to \cref{sec:THMB} for a detailed proof.
\medskip

\noindent\textbf{Acknowledgments.} G.A. has been partially supported by the NSF DMS Grant No. 2550590. G.A. wishes to thank M. Gromov for pointing out \cite[Conjecture 3, Section 6.1.2]{GromovFourLectures}. This project was initiated when G.A. was visiting G.F. and B.H. at the University of Augsburg.
G.A. is grateful for the hospitality and perfect working conditions. 
\smallskip

The authors used GPT-5.6 Sol Pro for help with language editing, TeX formatting, routine algebraic checking, and locating potentially relevant literature. All mathematical claims and proofs in this paper have been written by the authors and are entirely the authors' responsibility.

\section{Proof of \texorpdfstring{\cref{main:contractibility}}{Theorem A}(1)}\label{sec:A1}

In this section, we give the proof of \cref{main:contractibility}(1).
We shall assume that $n = 2$ and $\gamma>0$, or $n\ge3$ and $0<\gamma\leq \frac{4(n-1)}{n-2}$. We start by introducing the auxiliary space
\[ 
Z\coloneqq \{(g,u)\in\Riem^\gamma(M)\times C^\infty(M,(0,\infty))\mid -\gamma\Delta_gu + \scal_{g} u>0\}.
\]
We observe that the first positive and normalized (with $L^2$ norm $=1$) eigenfunction $u_g$ of the operator $L_g\coloneqq-\gamma\Laplace_g + \scal_g$ depends continuously on the metric $g$.
More precisely, for $g\in\Riem^\gamma(M)$ let $u_g$ be the unique function (see \cref{cor:Barta}) satisfying
\begin{equation}\label{eqn:UniqueNormalized}
  u_g>0, \qquad L_g u_g = \lambda_1(L_g)u_g, \qquad \int_M u_g^2\mathrm{d}\mathrm{vol}_g = 1.
\end{equation}
The following fact is well-known. 
A short proof will be provided in \cref{appendix:cont_dep}.

\begin{proposition}\label{prop:first-eigenfunction-cts-dependence}
  Let $\gamma>0$. For a closed manifold $M$, the assignment $\Riem^\gamma(M)\to C^\infty(M)$, $g\mapsto u_g$, is continuous.
\end{proposition}

Let $p:Z\to \Riem^\gamma(M)$ be defined as $p(g,u)\coloneqq g$. For every $g\in\Riem^\gamma(M)$, by \cref{cor:Barta} we have $(g,u_g)\in Z$. By the continuity of the map $g\mapsto u_g$ (see \cref{prop:first-eigenfunction-cts-dependence}), the map
\[
    s:\Riem^\gamma(M)\longrightarrow Z,
    \qquad
    s(g)\coloneqq (g,u_g),
\]
is continuous. Moreover, it is a section of $p$, that is,
$p\circ s=\mathrm{id}_{\Riem^\gamma(M)}$.

We claim that $s\circ p$ is homotopic to the identity on $Z$. Define the continuous map
\begin{equation}\label{eqn:KHomotopy}
    K:[0,1]\times Z\longrightarrow Z,
    \qquad
    K(t,(g,u))
    \coloneqq 
    \bigl(g,(1-t)u+tu_g\bigr).
\end{equation}
If $(g,u)\in Z$, then
both $u$ and $u_g$ are positive and therefore $
    (1-t)u+tu_g>0$ for every $t\in [0,1]$.
Moreover, by the linearity of $L_g$,
\[
    L_g\bigl((1-t)u+tu_g\bigr)
    =(1-t)L_gu+tL_gu_g
    =(1-t)L_gu+t\lambda_1(L_g)u_g
    >0.
\]
Thus, $K(t,(g,u))\in Z$ for every $t\in[0,1]$ and every $(g,u)\in Z$, proving the well-definedness of $K$. At the endpoints of the homotopy $K$, we have $
    K(0,\cdot)=\mathrm{id}_Z$
and $
    K(1,\cdot)=s\circ p$.
    
Hence, we conclude that $p$ is a homotopy equivalence, with homotopy inverse $s$. Thus, in order to complete the proof of \cref{main:contractibility}(1), it is enough to prove that
\begin{equation}\label{defPsi}
    \Psi:\Riem^0(M)\longrightarrow Z,
    \qquad
    \Psi(g)\coloneqq (g,1),
\end{equation}
is a homotopy equivalence. Indeed, if this is true, we have that the inclusion 
$$p\circ\Psi=\iota:\Riem^0(M)\hookrightarrow\Riem^\gamma(M)$$
is a homotopy equivalence, too, as desired.

\medskip

Let us prove that \eqref{defPsi} is a homotopy equivalence.
We claim that the map
\[
 \Phi \colon Z \to \Riem^{0}(M), \qquad (g,u)\mapsto u^{\frac\gamma{n-1}}g
\]
is well-defined and homotopy inverse to $\Psi$.

Recall the conformal change formula for the scalar curvature of the metric $\tilde g = e^{2\varphi}g$
\begin{equation}\label{eq:conformal-change-scalar-curvature}
  \scal_{\tilde g} = e^{-2\varphi}\left(\scal_g-2(n-1)\Delta_g\varphi-(n-2)(n-1)|\nabla_g \varphi|_g^2\right).
\end{equation}
For $\varphi = \tfrac{\gamma}{2(n-1)}\log(u)$ and thus $\tilde g = e^{2\varphi}g = u^{\gamma/(n-1)}g$ this becomes
\begin{align*}
  \scal_{\tilde g} ={}& u^{-\frac{\gamma}{n-1}}\left(\scal_g-\gamma\frac{\Laplace_g u}{u} + \gamma\frac{|\nabla_gu|^2}{u^2} -\frac{\gamma^2(n-2)}{4(n-1)}\frac{|\nabla_g u|_g^2}{u^2}\right)\\
    ={}&u^{-\frac{\gamma}{n-1}}\left(\frac{-\gamma\Laplace_gu + \scal_g\cdot u}{u} + \gamma\left(1-\frac{\gamma(n-2)}{4(n-1)}\right)\frac{|\nabla_g u|_g^2}{u^2}\right).
\end{align*}
When $n\geq 3$, since we have $0< \gamma\le\frac{4(n-1)}{n-2}$, the term $\left(1-\tfrac{\gamma(n-2)}{4(n-1)}\right)$ is nonnegative, and thus for $(g,u)\in Z$ we have the following
\begin{align*}
  \begin{split}
    \scal_{u^{\gamma/(n-1)}g} \ge{}& u^{-\frac{\gamma}{n-1}-1}(-\gamma\Delta_g u + \scal_g\cdot u) >0.
  \end{split}
\end{align*}  
We conclude that  $u^{\gamma/(n-1)}g\in \Riem^0(M)$, and $\Phi$ is indeed well-defined. Similarly, when $n=2$ and $\gamma\geq 0$ we have $u^{\gamma/(n-1)}g\in\Riem^0(M)$, and thus $\Phi$ is also well-defined in this case.

The composition $\Phi\circ\Psi$ is the identity on $\Riem^0(M)$.
To conclude the proof that $\Phi$ is homotopy inverse to $\Psi$, we claim that the identity on $Z$ is homotopic to the composition $\Psi\circ\Phi$ via the  homotopy\footnote{At the value $n=3,\gamma=8$, compare the following homotopy $H$ and \eqref{eqn:KHomotopy} with
\cite[Proof of Theorem 1.4]{LiMantoulidis2023}.}
\[H\colon [0,1]\times Z\to Z,\quad (t,(g,u))\mapsto \left((tu+1-t)^{\frac{\gamma}{n-1}}g,\ \frac{u}{tu+1-t}\right).
\]
It remains to show that $H$ is well-defined. For $t \in [0,1]$, $(g,u)\in Z$, we denote
\begin{equation}\label{eqn:ToShow}
   u_t \coloneqq  tu+1-t, \quad g_t \coloneqq  u_t^{\gamma/(n-1)}g,
\end{equation}
It is enough to show that
\begin{equation}\label{eqn:ToShow2}
-\gamma\Delta_{g_t}\left(\frac{u}{u_t}\right) + \scal_{g_t}\cdot \left(\frac{u}{u_t}\right)>0. 
\end{equation}
In this way, using \eqref{eqn:ToShow2} and \cref{cor:Barta}, we conclude that $H(t,(g,u))\in Z$ for every $t\in[0,1]$, and $(g,u)\in Z$, as desired.

Let us prove \eqref{eqn:ToShow2}. Since $\nabla_g u_t = t\nabla_gu$, we obtain the following
\begin{align*}
  \begin{split}
    \scal_{g_t} ={}& u_t^{-\frac{\gamma}{n-1}}\Biggl(\frac{-\gamma\Delta_g u_t + \scal_{g}u_t}{u_t} + \gamma\left(1-\frac{\gamma(n-2)}{4(n-1)}\right)\frac{|\nabla_g u_t|^2}{u_t^2}\Biggr)\\
      ={}& u_t^{-\frac{\gamma}{n-1}}\Biggl( \frac{-\gamma t\Delta_g u}{u_t} + \scal_{g} + \gamma\left(1-\frac{\gamma(n-2)}{4(n-1)}\right)\frac{t^2|\nabla_g u|^2}{u_t^2}\Biggr) \\
    \nabla_{g}\left(\frac{u}{u_t}\right) ={}& \frac{u_t\nabla_g u - u\nabla_g u_t}{u_t^2} =  \frac{u_t\nabla_g u - t u\nabla_g u}{u_t^2} = \frac{(1-t)}{u_t^2}\nabla_g u.
  \end{split}
\end{align*}
Furthermore,
\begin{align}
  \begin{split}
\Delta_{g}\left(\frac{u}{u_t}\right) ={}& (1-t) \, {\rm div}\left(\frac{1}{u_t^2}\nabla_g u\right) = (1-t)\left(\frac{1}{u_t^2}\Delta_g u - \frac{2\scpr{\nabla_gu,\nabla_gu_t}}{u_t^3}\right)\\
      ={}& (1-t)\left(\frac{1}{u_t^2}\Delta_g u - \frac{2t|\nabla_gu|^2}{u_t^3}\right).
  \end{split}
\end{align}
Recall the conformal change formula for the Laplacian of a metric $\tilde g = e^{2\varphi}g$: for every smooth function $\psi$ we have
\begin{equation}\label{eq:conformal-change-laplace-operator}
  \Laplace_{\tilde g}\psi = e^{-2\varphi}\left(\Laplace_g\psi + (n-2)\scpr{\nabla\varphi,\nabla\psi}\right).
\end{equation}
For $\varphi = \tfrac{\gamma}{2(n-1)}\log(u_t)$, thus $\tilde g = g_t$, this becomes
\begin{align*}
  \Laplace_{g_t}\psi = u_t^{-\frac{\gamma}{n-1}}\left( \Laplace_g\psi + \frac{\gamma(n-2)}{2(n-1)u_t}\scpr{\nabla_gu_t,\nabla_g\psi}\right)= u_t^{-\frac{\gamma}{n-1}}\left(\Laplace_g\psi + \frac{\gamma(n-2)t}{2(n-1)u_t}\scpr{\nabla_gu,\nabla_g\psi}\right),
\end{align*}
and therefore, substituting $\psi=u/u_t$, we obtain the following
\begin{align}
  \begin{split}
    -\gamma\Delta_{g_t}\left(\frac{u}{u_t}\right) ={}& -\gamma u_t^{-\frac{\gamma}{n-1}}\left(\Delta_g\left(\frac{u}{u_t}\right) + \frac{\gamma(n-2)t}{2(n-1)u_t}\scpr{\nabla_g u,\nabla_g \left(\frac{u}{u_t}\right)}\right)\\
      ={}& -\gamma u_t^{-\frac{\gamma}{n-1}}\left( (1-t)\left(\frac{1}{u_t^2}\Delta_g u - \frac{2t|\nabla_gu|^2}{u_t^3}\right) +  \frac{\gamma(n-2)}{4(n-1)}\frac{2t(1-t)}{u_t^3}|\nabla_gu|^2\right) \\
      ={}& u_t^{-\frac{\gamma}{n-1}}\left( \frac{-\gamma(1-t)\Delta_g u}{u_t^2} + \gamma\left(1-\frac{\gamma(n-2)}{4(n-1)}\right)\frac{2t(1-t)|\nabla_gu|^2}{u_t^3}\right).
  \end{split}
\end{align}
Hence,
\begin{align}
  \begin{split}
    -\gamma\Delta_{g_t}\left(\frac{u}{u_t}\right) &+ \scal_{g_t}\left(\frac{u}{u_t}\right)\\
      ={}&u_t^{-\frac{\gamma}{n-1}}\left( \frac{-\gamma(1-t)\Delta_g u}{u_t^2} + \gamma\left(1-\frac{\gamma(n-2)}{4(n-1)}\right)\frac{2t(1-t)|\nabla_gu|^2}{u_t^3}\right)  \\
      &+u_t^{-\frac{\gamma}{n-1}}\left( \frac{-\gamma tu\Delta_g u}{u_t^2} + \frac{\scal_{g}u}{u_t} + \gamma\left(1-\frac{\gamma(n-2)}{4(n-1)}\right)\frac{t^2u|\nabla_g u|^2}{u_t^3}\right)\\
      ={}&u_t^{-\frac{\gamma}{n-1}}
        \Biggl( -\gamma\underbrace{\left(1-t + tu\right)}_{=u_t}\frac{\Delta_g u}{u_t^2} + \frac{\scal_{g}u}{u_t}\\
        &\qquad\qquad+ \gamma \underbrace{\left(1-\frac{\gamma(n-2)}{4(n-1)}\right)}_{\ge0}\underbrace{(2t(1-t)+t^2 u)}_{\ge0}\frac{|\nabla_gu|^2}{u_t^3}\Biggr) \\
      \ge{}&u_t^{-\frac{\gamma}{n-1}-1}
      \big(-\gamma \Delta_g u +\scal_{g}u \big) >0,
  \end{split}
\end{align}
thus proving \eqref{eqn:ToShow2}.
The last inequality is true because $u$ is a positive smooth function satisfying $-\gamma \Delta_g u +\scal_{g}u>0$.
Thus, $H(t,(g,u))\in Z$ for all $t\in[0,1]$ and all $(g,u)\in Z$.
We conclude that $H$ is a well-defined homotopy. 
This completes the proof of \cref{main:contractibility}(1).
\hfill$\Box$

\section{Proof of \texorpdfstring{\cref{main:contractibility}}{Theorem A}(2)} 

The proof of \cref{main:contractibility}(2) is carried out in two steps. 
\begin{itemize}
    \item First, we translate the global condition $g\in\Riem^\gamma(M)$ into a local one. This is done in \cref{sec:loctoglob}.
    \item Second, in \cref{sec:qc-contractible} we prove that the localized version $\Q_c(M)$ of $\Riem^\gamma(M)$, defined as in \eqref{eqn:QcmYeah}, is contractible.
\end{itemize}

\subsection{Global to local}\label{sec:loctoglob}
One of the difficulties in working with $\Riem^\gamma(M)$ is the fact that the condition $-\gamma\Laplace_g+\scal_g>0$ is not a local condition but rather a global condition.
To circumvent this, we replace $\Riem^\gamma(M)$ with a local counterpart.
More precisely, we define for $c>0$ the space
\begin{equation}\label{eqn:QcmYeah}
 \Q_c(M)\coloneqq \{(h,\varphi)\in \Riem(M)\times C^\infty(M)\mid \scal_h+c|\nabla_h\varphi|_h^2>0\}.
\end{equation}
We prove the following result in \cref{sec:qc-contractible}.
\begin{theorem}\label{thm:main}
Let $M$ be a closed smooth manifold of dimension $n\geq 2$. For every $c>0$, the space $\Q_c(M)$ is contractible.
\end{theorem}
  
We first explain how \cref{thm:main} implies the contractibility of $\Riem^\gamma(M)$.
To this end, we recall the conformal transformation formula for the scalar curvature and the Laplace operator.
If $g=e^{2\varphi}h$, then
\begin{align}\label{eq:conformal-scalar}
\scal_{g}={}&e^{-2\varphi}\Bigl(\scal_h-2(n-1)\Delta_h \varphi-(n-1)(n-2)|\nabla_h\varphi|_h^2\Bigr)\\
\Delta_g\cdot ={}&e^{-2\varphi}\left(\Delta_h\cdot+(n-2)h(\nabla_h\cdot,\nabla_h\varphi)\right).
\end{align}
We observe that for $(h,\varphi)\in\Riem(M)\times C^\infty(M)$ and $a\in\R$ we have 
\begin{align*}
-&\gamma\Delta_g(e^{-a\varphi})+\scal_g e^{-a\varphi} \\
&=-\gamma e^{-2\varphi}\left(\Delta_h(e^{-a\varphi})+(n-2)h(\nabla_h (e^{-a\varphi}),\nabla_h\varphi)\right)+\scal_ge^{-a\varphi} \\
&=e^{-(2+a)\varphi}\left(\scal_h+(\gamma a-2(n-1))\Delta_h\varphi+(-\gamma a^2+a\gamma(n-2)-(n-2)(n-1))|\nabla_h\varphi|_h^2\right).
\end{align*}
Thus, if we choose $a\coloneqq\frac{2(n-1)}{\gamma}$, the term with $\Delta_h\varphi$ cancels out, and we obtain
\begin{align*}
-\gamma\Delta_g(e^{-a\varphi})+\scal_ge^{-a\varphi} = e^{-(2+a)\varphi}\left(\scal_h+\left(-\frac{4(n-1)^2}{\gamma}+(n-1)(n-2)\right)|\nabla_h\varphi|_h^2\right).
\end{align*}
Note that 
\begin{equation}\label{eqn:cngamma}
c_{n,\gamma}\coloneqq-\frac{4(n-1)^2}{\gamma}+(n-1)(n-2)>0,
\end{equation}
because $\gamma>\gamma_n$.
Hence, simplifying, we get the key identity
\begin{equation}\label{eqn:KeyRgammaMtoQcM}
-\gamma\Delta_g(e^{-a\varphi})+\scal_g e^{-a\varphi} = e^{-(2+a)\varphi}\left(\scal_h + c_{n,\gamma}|\nabla_h\varphi|^2\right),
\end{equation}
which finally gives the following
\[
(h,\varphi)\in \calQ_{c_{n,\gamma}} \Rightarrow e^{2\varphi}h\in\Riem^\gamma(M),
\]
 using \eqref{eqn:KeyRgammaMtoQcM} and \cref{cor:Barta}. 
We have thus constructed a continuous map 
\[
\Psi\colon\Q_{c_{n,\gamma}}(M)\to\Riem^\gamma(M),\qquad (h,\varphi)\mapsto e^{2\varphi}h.
\]
In order to construct a continuous map from  $\Riem^\gamma(M)$ to $\Q_{c_{n,\gamma}}(M)$, we consider the normalized first eigenfunction of $L_{g,\gamma}\coloneqq-\gamma\Delta_g+\scal_g$.
If $g\in \Riem^\gamma(M)$ let $u_g$ be the normalized first eigenfunction of $L_g$ (see \eqref{eqn:UniqueNormalized}).
We note that the assignment $g\mapsto u_g$ is continuous by \cref{prop:first-eigenfunction-cts-dependence}.

Writing $\varphi\coloneqq-\tfrac{\gamma}{2(n-1)} \log(u_g) = -\tfrac{1}{a}\log(u_g)$ and $h\coloneqq e^{-2\varphi} g$, \eqref{eqn:KeyRgammaMtoQcM} gives
\begin{align*}
  e^{-(2+a)\varphi}\left(\scal_{h} + c_{n,\gamma}|\nabla_h\varphi |_{h}^2\right) ={}& -\gamma\Delta_g(e^{-a\varphi})+\scal_ge^{-a\varphi}\\
    ={}& -\gamma\Delta_g(u_g)+\scal_gu_g = \lambda_1(L_g) u_g >0.
\end{align*} 
Therefore, we obtain a well-defined continuous map 
\[
\Xi\colon\Riem^\gamma(M)\to\Q_{c_{n,\gamma}}(M),\qquad g\mapsto \left(u_g^{\frac{\gamma}{n-1}} g,\ -\frac{\gamma}{2(n-1)}\log(u_g)\right).
\]
Finally, we note that
\[
\Psi(\Xi(g)) = \Psi\left(u_g^{\frac{\gamma}{n-1}} g,\  -\frac{\gamma}{2(n-1)}\log(u_g)\right) = e^{-2\frac{\gamma}{2(n-1)}\log(u_g)}u_g^{\frac{\gamma}{n-1}} g = g,
\]
which implies that $\Psi\circ\Xi={\rm id}_{\Riem^\gamma(M)}$.

By \cref{thm:main}, the space $\Q_c(M)$ is contractible and thus $\Q_c(M)$ is nonempty and every self-map of $\Q_c(M)$ is homotopic to the identity,
in particular, so is $\Xi\circ\Psi$.
Thus, $\Xi$ and $\Psi$ are mutual homotopy inverses. Hence, $\Riem^\gamma(M)$ is contractible. This proves \cref{main:contractibility}(2).
    
\subsection{The space \texorpdfstring{$\Q_c(M)$}{QcM} is contractible}\label{sec:qc-contractible}
We now begin the proof of \cref{thm:main}.
This proof will rely on two propositions, which we prove in \cref{sec:local-scalar} and \cref{sec:parametric-input}. Let $k\geq 0$.
Consider the following extension problem
\[\begin{tikzcd}
    \mathbb S^{k-1} \arrow[rr, "{(g,\varphi)}"] \arrow[d, hook]
      & & \Q_c(M) \\
    \mathbb D^k \arrow[rru,dashed] 
\end{tikzcd}\]
that is, we aim to show that every continuous map $(g,\varphi)\colon \mathbb S^{k-1}\to\Q_c(M)$, $\xi \mapsto(g_\xi,\varphi_\xi)$, can be extended to a continuous map $\mathbb D^k\to\mathcal{Q}_c(M)$.
This will prove weak contractibility, as it implies that all homotopy groups are trivial.
Hence, \cite[Theorem~15, p.~5]{Palais1966}, which asserts that a metrizable manifold $X$ is
contractible if and only if it is weakly contractible, finishes the proof.
We note that $\mathbb S^{-1}=\varnothing$ and hence for $k=0$, the existence of such an extension above proves that the space $\Q_c(M)$ is nonempty.
\medskip

We first note that, when $k>0$, in the above diagram it is enough to extend smooth maps $\mathbb S^{k-1}\to\mathcal{Q}_c(M)$. Indeed,
let $$F:\mathbb S^{k-1}\to \mathcal{Q}_c(M)\subset\mathcal{R}(M)\times C^\infty(M)$$ be a continuous map $\xi \mapsto (g_\xi,\varphi_\xi)$. 
By smoothing in the parameter $\xi$, the map $F$ can be
approximated, uniformly in $\xi$,  by a smooth map $\widetilde F:\mathbb S^{k-1}\to \mathcal{R}(M)\times C^\infty(M)$, $\xi\mapsto (\widetilde g_\xi,\widetilde\varphi_\xi)$.

Since $(g,\varphi)\mapsto \scal_g+c|\nabla_g\varphi|_g^2$ is continuous, compactness gives $$\eta:=\min_{(\xi,m)\in\mathbb S^{k-1}\times M}\bigl(\scal_{g_\xi}(m)+c|\nabla_{g_\xi}\varphi_\xi|_{g_\xi}^2(m)\bigr)>0.$$
Choosing $\widetilde F$ sufficiently close to $F$, if we set $g_{\xi,s}:=(1-s)g_\xi+s\widetilde g_\xi$ and $\varphi_{\xi,s}:=(1-s)\varphi_\xi+s\widetilde\varphi_\xi$ for $(\xi,s)\in\mathbb S^{k-1}\times [0,1]$, we have $$\scal_{g_{\xi,s}}(m)+c|\nabla_{g_{\xi,s}}\varphi_{\xi,s}|_{g_{\xi,s}}^2(m)\geq \eta/2,\quad \textrm{ for all } (\xi,s)\in \mathbb S^{k-1}\times [0,1],\, m\in M.$$
Thus, $(\xi,s)\mapsto (g_{\xi,s},\varphi_{\xi,s})$ is a homotopy in $\mathcal Q_c(M)$ from $F$ to the smooth map $\widetilde F$. 
Hence, from now on we may assume without loss of generality that the map $(g,\varphi):\mathbb S^{k-1}\to \mathcal Q_c(M)$ we aim to extend is smooth.
\medskip

We start by fixing once and for all a Morse function $f\colon M\to\R$ and a metric $g_0$.
After applying the following proposition to the case $X=\{\pt\}$, $X_0=\varnothing$ and $\psi_\pt=f$, we may assume that there exists a neighborhood $U\subset M$ of the critical points $\Crit(f)$ of $f$ such that $\scal_{g_0}>0$ on $U$. Thus (we will prove this below in \eqref{eq:basepoint}) there exists a sufficiently large constant $\Lambda_0\geq 1$ such that $(g_0,\Lambda_0f)\in \mathcal{Q}_c(M)$. Hence, $\mathcal{Q}_c(M)\neq\varnothing$.
\begin{proposition}\label{lem:local-improvement}
Let $X$ be a compact smooth manifold, possibly with boundary, and let
$X_0\subset X$ be a closed subset. Let $g_x$, $x\in X$, be a smooth family of Riemannian metrics on $M^{n\geq 2}$,
and let $\psi_x$, $x\in X$, be a smooth fiberwise finite-critical family of functions (namely, for every $x\in X$, $\psi_x$ has only finitely many critical points). Put
\[
 C=\{(x,m)\in X\times M\mid\mathrm{d}\psi_x(m)=0\}.
\]
Assume that
\begin{equation}\label{eq:positive-on-relative-critical-set}
 \scal_{g_x}(m)>0
 \qquad
 \text{for every }(x,m)\in C\text{ with }x\in X_0.
\end{equation}
Then there is a smooth family of metrics $G_x$, an open neighborhood $\Omega\subset X\times M$ of
$C$, and an open neighborhood $U_0\subset X$ of $X_0$, such that
\begin{enumerate}[label=\textup{(\roman*)}]
\item $G_x=g_x$ for all $x\in U_0$;
\item $\scal_{G_x}>0$ on $\Omega_x\coloneqq \{m\in M\mid(x,m)\in\Omega\}$ for every $x\in X$.
\end{enumerate}
\end{proposition}
\noindent The proof of the proposition above consists of an appropriate family of conformal changes of $g_x$ supported away from $X_0$. 
It is given below in \cref{sec:local-scalar}.

\medskip

On $M\setminus U$, the one-form $\mathrm{d}f$ is nowhere zero. If $M\setminus U$ is nonempty, define
\[
 \mu\coloneqq \min_{M\setminus U}|\nabla_{g_0}f|_{g_0}^2>0;
\]
if $M\setminus U=\varnothing$, choose any $\mu>0$. Set
\[
 B\coloneqq \max\left\{0,-\min_M \scal_{g_0}\right\}.
\]
Choose $\Lambda_0\geq1$ so large that
\begin{equation}\label{eq:Lambda-star-choice}
 c\Lambda_0^2\mu>B.
\end{equation}
Then
\[
 \scal_{g_0}+c|\nabla_{g_0}(\Lambda_0 f)|_{g_0}^2>0
\]
on $M\setminus U$ by \eqref{eq:Lambda-star-choice}, and because of $\scal_{g_0}>0$ on $U$.
Thus,
\begin{equation}\label{eq:basepoint}
 (g_0,\Lambda_0 f)\in\Q_c(M).
\end{equation}
In particular, the space $\Q_c(M)$ is nonempty and, from now on we may assume $k>0$ and hence $\mathbb S^{k-1}\not=\varnothing$.
The pair $(g_0,\Lambda_0 f)$ is the base point to which all compact families will be contracted to.
\smallskip

By the compactness of $\mathbb S^{k-1}$ and $M$, there is a positive margin
\begin{equation}\label{eq:initial-margin}
 \varepsilon_0\coloneqq \min_{(\xi,m)\in\mathbb S^{k-1}\times M}\bigl(\scal_{g_\xi}(m)+c|\nabla_{g_\xi}\varphi_\xi|_{g_\xi}^2(m)\bigr)>0.
\end{equation}
We consider the straight line homotopy
\[
H_{\xi,t}\coloneqq (1-t)\varphi_\xi + tf.
\]
We have the following proposition, which states that this homotopy can be smoothly approximated by functions which have finitely many critical points.
Its proof, which we present in \cref{sec:parametric-input}, relies on a foliated genericity theorem of Bertelson, see \cref{thm:bertelson-isolated}.

\begin{proposition}\label{cor:finite-critical-path}
Let $X$ be a closed manifold.
Let
$\varphi_x$, $x\in X$, be a smooth family of functions in $C^\infty(M)$, and let $f\in C^\infty(M)$ be fixed.
For every neighborhood of the linear interpolation
\[
 H_{x,t}\coloneqq (1-t)\varphi_x+t f,
 \qquad (x,t)\in X\times[0,1],
\]
there exists a smooth family
\[
 \psi_{x,t}\in C^\infty(M),
 \qquad (x,t)\in X\times[0,1],
\]
inside that neighborhood such that every function $\psi_{x,t}$ has finitely many critical points.
\end{proposition}

Now, let $\psi_{\xi,t}$ be a smooth approximation of $H_{\xi,t}$ such that $\psi_{\xi,t}$ only has finitely many critical points for every pair $(\xi,t)\in\mathbb S^{k-1}\times[0,1]$.
Since this approximation may be chosen to be contained in any neighborhood of $H_{\xi,t}$, we may assume that $\psi_{\xi,t}$ satisfies the following three conditions:
\begin{enumerate}
  \item For all $\xi \in \mathbb S^{k-1}$, the straight-line segment $\theta^{0}_{\xi ,s}\coloneqq (1-s)\varphi_\xi+s\psi_{\xi,0}$ for $0\leq s\leq1$ satisfies
  \begin{equation}\label{eq:initial-short-positive}
    \scal_{g_\xi}+c|\nabla_{g_\xi}\theta^{0}_{\xi,s}|_{g_\xi}^2>0\qquad\text{for all }(\xi,s)\in \mathbb S^{k-1}\times[0,1].
  \end{equation}
  This can be achieved as we are dealing with the strict inequality, which gives an open condition, and compactness of $\mathbb S^{k-1}\times M$.
  Moreover, we require that at every critical point of $\psi_{\xi,0}$,
  \begin{equation}\label{eq:initial-critical-R-positive}
  \scal_{g_\xi}>0.
  \end{equation}
  Indeed, if $\mathrm{d}\psi_{\xi,0}=0$ and $\psi_{\xi,0}$ is $C^1$-close enough to $\varphi_\xi$, then
  $|\nabla_{g_\xi}\varphi_\xi|_{g_\xi}^2<\varepsilon_0/(2c)$, so \eqref{eq:initial-margin} gives
  $\scal_{g_\xi}>\varepsilon_0/2$.
  \medskip
  \item Second, the terminal endpoint $\psi_{\xi,1}$ is chosen so that $\Crit(\psi_{\xi,1})\subset U$ for every $\xi\in \mathbb S^{k-1}$, where $U$ is the open neighborhood of $\Crit(f)$ fixed above.
  This is possible because $|\nabla_{g_0}f|_{g_0}$ has a positive minimum on the compact set $M\setminus U$.
  Since $\scal_{g_0}>0$ on $U$, the fact that $\Crit(\psi_{\xi,1})\subset U$ implies
  \begin{equation}\label{eq:terminal-critical-R-positive}
    \scal_{g_0}>0\quad\text{at every critical point of }\psi_{\xi,1}.
  \end{equation}
  \item Third, for the straight line homotopy from $\psi_{\xi,1}$ back to $f$ given by $\theta^{1}_{\xi,s}\coloneqq (1-s)\psi_{\xi,1}+s f$, we require that, on $M\setminus U$,
  \begin{equation}\label{eq:terminal-gradient-bound}
    4|\nabla_{g_0}\theta^{1}_{\xi,s}|_{g_0}^2\geq \mu \qquad\text{for all }(\xi,s)\in \mathbb S^{k-1}\times[0,1].
  \end{equation}
  This follows if $\mathrm{d}\psi_{\xi,1}$ is uniformly sufficiently close to $\mathrm{d}f$. 
\end{enumerate}

Next, we define the convex interpolation of metrics
\begin{equation}\label{eq:metric-interpolation}
 \overline g_{\xi,t}\coloneqq (1-t)g_\xi+t g_0.
\end{equation}
This is a smooth family of Riemannian metrics. 
Let
\[
 X=\mathbb S^{k-1}\times[0,1], \qquad X_0=\mathbb S^{k-1}\times\{0,1\}.
\]
By \eqref{eq:initial-critical-R-positive} and \eqref{eq:terminal-critical-R-positive}, the family
$\overline g_{\xi,t}$ has positive scalar curvature at the critical points of $\psi_{\xi,t}$ over $X_0$.
Therefore, \cref{lem:local-improvement} applies. 
It gives a smooth family of metrics $G_{\xi,t}$, equal to $\overline g_{\xi,t}$ near $t=0$ and near $t=1$, and an open neighborhood
\[
 \Omega\subset \mathbb S^{k-1}\times[0,1]\times M
\]
of $C=\{(\xi,t,m)\in \mathbb S^{k-1}\times[0,1]\times M \mid \mathrm{d}\psi_{\xi,t}(m)=0\}$ such that
\begin{equation}\label{eq:G-positive-on-Omega}
 \scal_{G_{\xi,t}}>0\quad\text{on }\Omega_{\xi,t},
\end{equation}
where we recall that $\Omega_{\xi,t}= \{m\in M \mid (\xi,t,m)\in\Omega\}$. On the compact complement of $\Omega$, the one-form $\mathrm{d}\psi_{\xi,t}$ is nowhere zero. 
Hence,
\begin{equation}\label{eq:middle-gradient-lower}
 \mu_0\coloneqq \min_{(\xi,t,m)\in (\mathbb S^{k-1}\times[0,1]\times M)\setminus\Omega}|\nabla_{G_{\xi,t}}\psi_{\xi,t}|_{G_{\xi,t}}^2(m)>0,
\end{equation}
with the convention that if the complement is empty, any positive $\mu_0$ may be chosen. 
Also put
\begin{equation}\label{eq:middle-scalar-lower}
 A_0\coloneqq \max\left\{0,-\min_{(\xi,t,m)\in\mathbb S^{k-1}\times[0,1]\times M}\scal_{G_{\xi,t}}(m)\right\}.
\end{equation}
Choose a constant $\Lambda\geq \Lambda_0\geq 1$ so large that
\begin{equation}\label{eq:Lambda-choice}
 c\Lambda^2\mu_0>A_0,
 \qquad
 c\Lambda^2\mu>4B.
\end{equation}
Then we claim that
\begin{equation}\label{eq:middle-in-Q}
 (G_{\xi,t},\Lambda\psi_{\xi,t})\in\Q_c(M)
 \qquad\text{for every }(\xi,t)\in \mathbb S^{k-1}\times[0,1].
\end{equation}
Indeed, on $\Omega_{\xi,t}$ this follows from \eqref{eq:G-positive-on-Omega}; outside $\Omega_{\xi,t}$ it follows
from \eqref{eq:middle-gradient-lower}, \eqref{eq:middle-scalar-lower}, and \eqref{eq:Lambda-choice}.
We now concatenate five homotopies in $\Q_c(M)$.
\begin{enumerate}[label=\textup{(\arabic*)}]
\item The initial homotopy
\[
 (g_\xi,\varphi_\xi)\leadsto(g_\xi,\psi_{\xi,0})
\]
is given by $\theta^{0}_{\xi,s}\coloneqq (1-s)\varphi_\xi+s\psi_{\xi,0}$; it lies in $\Q_c(M)$ by
\eqref{eq:initial-short-positive}.

\item Scale the initial endpoint:
\[
 (g_\xi,\psi_{\xi,0})\leadsto(g_\xi,\Lambda\psi_{\xi,0}).
\]
For $a\in[1,\Lambda]$, we have
\[
 \scal_{g_\xi}+c|\nabla_{g_\xi}(a\psi_{\xi,0})|_{g_\xi}^2
 =\bigl(\scal_{g_\xi}+c|\nabla_{g_\xi}\psi_{\xi,0}|_{g_\xi}^2\bigr)
  +c(a^2-1)|\nabla_{g_\xi}\psi_{\xi,0}|_{g_\xi}^2>0,
\]
because $(g_\xi,\psi_{\xi,0})\in\Q_c(M)$ by step (1).

\item Use the homotopy \eqref{eq:middle-in-Q}:
\[
 (g_\xi,\Lambda\psi_{\xi,0})\leadsto(g_0,\Lambda\psi_{\xi,1}).
\]
Here we used that $G_{\xi,t}=\overline g_{\xi,t} = (1-t)g_\xi+t g_0$ near the endpoints, so $G_{\xi,0}=g_\xi$ and $G_{\xi,1}=g_0$.

\item Move the terminal function from $\psi_{\xi,1}$ to $f$ while keeping the metric $g_0$ and the scale $\Lambda$:
\[
 (g_0,\Lambda\psi_{\xi,1})\leadsto(g_0,\Lambda f).
\]
The path is $(g_0,\Lambda\theta^{1}_{\xi,s})$. 
On $U$, positivity follows from $\scal_{g_0}>0$.
On $M\setminus U$, it follows from \eqref{eq:terminal-gradient-bound} and \eqref{eq:Lambda-choice}:
\[
 \scal_{g_0}+c\Lambda^2|\nabla_{g_0}\theta^{1}_{\xi,s}|_{g_0}^2
 \geq -B+c\Lambda^2\frac{\mu}{4}>0,
\]
where $B$ and $\mu$ are defined below \cref{lem:local-improvement}.

\item Finally, scale $\Lambda f$ down to $\Lambda_0f$:
\[
 (g_0,\Lambda f)\leadsto(g_0,\Lambda_0f).
\]
For every $\lambda\in[\Lambda_0,\Lambda]$, the pair $(g_0,\lambda f)$ lies in $\Q_c(M)$: on $U$ we
have $\scal_{g_0}>0$, and on $M\setminus U$, by \eqref{eq:Lambda-star-choice},
\[
 \scal_{g_0}+c\lambda^2|\nabla_{g_0}f|_{g_0}^2
 \geq -B+c\Lambda_0^2\mu>0.
\]
\end{enumerate}
The concatenation of the above five homotopies exhibits a $\mathcal{Q}_c(M)$-valued homotopy connecting the family $(g_\xi,\varphi_\xi)$, $\xi\in\mathbb S^{k-1}$, to the constant family $(g_0,\Lambda_0f)$. Hence, this yields the sought extension $\mathbb D^k\to \mathcal{Q}_c(M)$. 
This proves \cref{thm:main}, assuming \cref{lem:local-improvement} and \cref{cor:finite-critical-path}.

\subsubsection{A local scalar-curvature deformation near finite critical sets}\label{sec:local-scalar}

We next prove \cref{lem:local-improvement}.
The key idea is the following: by a conformal change $e^{2u}g$, with $u$ locally equal to a negative quadratic polynomial, one can make the term $-\Delta_g u$ as large as needed near a prescribed finite set while keeping $|\nabla_g u|_g^2$ small relative to $-\Delta_g u$.
This is achieved by working in sufficiently small coordinate balls.
The mechanism of the proof of \cref{lem:local-improvement} is reminiscent of the one in \cite[Lemma 2.5]{BaerHankeLocalFlexibility}. 
We recall that $A\Subset B$ means that the closure $\overline A$ is compact and contained in $B$.

\begin{lemma}\label{lem:local-bump}
Let $K$ be a compact Hausdorff space, let $g_\kappa$,  $\kappa\in K$, be a continuous family of smooth metrics
on $M$, and let $z\in M$. Given constants $L>0$ and $\eta>0$, there are coordinate balls
\[
 B^-\Subset B^+\Subset M
\]
centered at $z$ and a function $u\in C_c^\infty(B^+)$ such that, for every $\kappa\in K$ and every
$x\in B^-$,
\begin{equation}\label{eq:bump-estimates}
 -\Delta_{g_k}u(x)\geq L,
 \qquad
 |\nabla_{g_\kappa}u(x)|_{g_\kappa}^2\leq \eta\,\bigl(-\Delta_{g_\kappa}u(x)\bigr).
\end{equation}
\end{lemma}

\begin{proof}
We work in a coordinate chart centered at $z$. On a small coordinate ball, take
\[
 u(y)=-A|y|^2,
\]
for some $A>0$ to be determined below.
For the Euclidean metric, $-\Delta u=2nA$ and $|\nabla u|^2=4A^2|y|^2$.
Since the parameter space $K$ is compact and the coefficients of $g_\kappa$ and their first derivatives vary continuously, after choosing the
coordinate radius $r$ sufficiently small (and $B^+$ being the coordinate ball centered at $z$ with radius $r$) the same estimates hold uniformly up to fixed multiplicative
constants:
\[
 -\Delta_{g_\kappa}u\geq C_1A,
 \qquad
 |\nabla_{g_\kappa}u|_{g_\kappa}^2\leq C_2A^2r^2
\]
on the ball $B^-$ of coordinate radius $r/2$ centered at $z$, with constants $C_1,C_2>0$ independent of $\kappa$ and $r\leq 1$.
Choose $A$ so large that $C_1A\geq L$, and then choose $r$ so small that $C_2A^2r^2\leq \eta C_1A$. Finally, multiply $u$ by a cutoff which is identically $1$ on $B^-$ and supported in $B^+$. 
The cutoff does not change the estimates on the inner ball.
\end{proof}

\begin{proof}[Proof of \cref{lem:local-improvement}]
The set $C$ is closed in the compact space $X\times M$, hence compact. By
\eqref{eq:positive-on-relative-critical-set} and continuity, there exists an open neighborhood $W$ of
\[
 C_0\coloneqq C\cap(X_0\times M)
\]
in $X\times M$ such that $\scal_{g_x}>0$ on $W$. Since $C$ is compact and $X_0$ is closed, we can find an open neighborhood $U_0^+$ of $X_0$ in $X$ such that
\begin{equation}\label{eq:positive-near-X0}
 C\cap(U_0^+\times M)\subset W.
\end{equation}
Thus, for parameters $x$ in $U_0^+$, the scalar curvature $\scal_{g_x}$ is positive at all critical points of $\psi_x$.
Let
\begin{equation}\label{eq:def-of-rho-0}
  \rho_0\coloneqq \min_{X\times M} \scal_{g_x}(m).
\end{equation}
We first choose a slightly smaller neighborhood on which the final deformation will vanish. Choose an
open neighborhood \(U_0\) of \(X_0\) such that
\[
 X_0\subset U_0 \Subset U_0^+ .
\]
We shall use \cref{lem:local-bump} with the following addition: if \(O\) is a
coordinate neighborhood of the point \(z\), then the balls produced by the lemma may be required to satisfy
\[
 B^-\Subset B^+\Subset O.
\]
Choose \(\eta>0\) so small that
\begin{equation}\label{eq:eta-choice}
 a\coloneqq 2-(n-2)\eta>0.
\end{equation}
Then choose \(L>0\) so large that
\begin{equation}\label{eq:L-choice}
 \rho_0+(n-1)aL>0,
\end{equation}
for $\rho_0$ as defined in \eqref{eq:def-of-rho-0}.
We now cover the relevant parameter space \(X\setminus U_0\). Fix \(x_*\in X\setminus {U_0}\), and
choose an open neighborhood \(Q_*\subset X\) of \(x_*\). Write
\[
 \Crit(\psi_{x_*})=\{z_{*,1},\ldots,z_{*,N_*}\},
\]
where \(N_*\) is allowed to be zero. If \(N_*=0\), then \(\mathrm d\psi_{x_*}\) is nonzero on all of \(M\). Hence, after taking an open
neighborhood \(V_*\Subset Q_*\) of \(x_*\), we have $
 \Crit(\psi_x)=\varnothing$
 for all $x\in V_*$.
In this case no bump functions are chosen.

Assume now that \(N_*>0\). Choose pairwise disjoint coordinate neighborhoods
\[
 O_{*,j}\ni z_{*,j},
 \qquad j=1,\ldots,N_* .
\]
Applying \cref{lem:local-bump}, in the localized form just described, to the compact parameter
space \(\overline{Q_*}\), to the family \(g_x\), \(x\in \overline{Q_*}\), and to the point \(z_{*,j}\), we obtain
coordinate balls
\[
 B_{*,j}^-\Subset B_{*,j}^+\Subset O_{*,j}
\]
and functions
\[
 u_{*,j}\in C_c^\infty(B_{*,j}^+)
\]
such that, for every \(x\in\overline{Q_*}\) and every \(m\in B_{*,j}^-\),
\[
 -\Delta_{g_x}u_{*,j}(m)\geq L,
 \qquad
 |\nabla_{g_x}u_{*,j}(m)|_{g_x}^2
 \leq
 \eta\bigl(-\Delta_{g_x}u_{*,j}(m)\bigr).
\]
Since the neighborhoods \(O_{*,j}\) are pairwise disjoint, the outer coordinate balls \(B_{*,j}^+\) are pairwise
disjoint as well.

Now, after the inner balls \(B_{*,j}^-\) have been fixed, use continuity of the family
\(\mathrm d\psi_x\) to obtain an open subset \(V_*\Subset Q_*\) so that
\[
 \Crit(\psi_x)\subset \bigcup_{j=1}^{N_*}B_{*,j}^-
 \qquad
 \text{for all }x\in V_*.
\]

Choose finitely many of these neighborhoods \(V_\alpha\), \(\alpha=1,\ldots,A\), covering the compact
set \(X\setminus U_0\). 
Thus, for every \(\alpha\), we have
\begin{equation}\label{eq:crit-stays-inner}
 \Crit(\psi_x)\subset\bigcup_{j=1}^{N_\alpha}B_{\alpha,j}^-
 \qquad
 \text{for all }x\in V_\alpha,
\end{equation}
with the convention that the right-hand side is empty if \(N_\alpha=0\). Moreover, for every
\(j=1,\ldots,N_\alpha\), every \(x\in\overline{Q_\alpha}\), and every \(m\in B_{\alpha,j}^-\),
\begin{equation}\label{eq:local-u-estimates}
 -\Delta_{g_x}u_{\alpha,j}(m)\geq L,
 \qquad
 |\nabla_{g_x}u_{\alpha,j}(m)|_{g_x}^2
 \leq
 \eta\bigl(-\Delta_{g_x}u_{\alpha,j}(m)\bigr).
\end{equation}

Put
\[
 u_\alpha\coloneqq \sum_{j=1}^{N_\alpha}u_{\alpha,j},
\]
with the convention that \(u_\alpha=0\) if \(N_\alpha=0\). Because the supports of the
\(u_{\alpha,j}\) are pairwise disjoint for fixed \(\alpha\), \eqref{eq:crit-stays-inner} and
\eqref{eq:local-u-estimates} imply that, for every \(x\in V_\alpha\) and every
\(m\in\Crit(\psi_x)\),
\begin{equation}\label{eq:u-alpha-estimates}
 -\Delta_{g_x}u_\alpha(m)\geq L,
 \qquad
 |\nabla_{g_x}u_\alpha(m)|_{g_x}^2
 \leq
 \eta\bigl(-\Delta_{g_x}u_\alpha(m)\bigr).
\end{equation}
Now, choose a smooth function \(\chi:X\to[0,1]\) such that
\[
 \chi=0 \quad\text{on } U_0,
 \qquad
 \chi=1 \quad\text{on } X\setminus U_0^+.
\]
Choose a smooth partition of unity
\[
 \{\mu_-,\mu_1,\ldots,\mu_A\}
\]
subordinate to the open cover
\[
 U_0,\,V_1,\ldots,V_A
\]
of \(X\), and set
\[
 \lambda_\alpha\coloneqq \chi\mu_\alpha,
 \qquad \alpha=1,\ldots,A.
\]
Then
\[
 0\leq\sum_\alpha\lambda_\alpha=\chi\leq1,
\]
all \(\lambda_\alpha\) vanish on \(U_0\), and
\[
 \sum_\alpha\lambda_\alpha=1
 \qquad
 \text{on } X\setminus U_0^+ .
\]
Define
\[
 U(x,m)\coloneqq \sum_{\alpha=1}^A\lambda_\alpha(x)u_\alpha(m),
 \qquad
 G_x\coloneqq e^{2U_x}g_x.
\]
Then \(G_x=g_x\) for all \(x\in U_0\).

It remains to verify the scalar curvature estimate of $G_x$, so let \((x,m)\in C\). For every index \(\alpha\) with \(\lambda_\alpha(x)>0\), we have
\(x\in V_\alpha\), so \eqref{eq:u-alpha-estimates} applies. Hence,
\[
 -\Delta_{g_x}U_x(m)
 =
 \sum_\alpha\lambda_\alpha(x)\bigl(-\Delta_{g_x}u_\alpha(m)\bigr)
 \geq 0.
\]
Moreover, since \(0\leq\sum_\alpha\lambda_\alpha(x)\leq1\), the convexity of the squared norm gives
\begin{align}\label{eq:sum-gradient}
 |\nabla_{g_x}U_x(m)|_{g_x}^2
 &=
 \left|\sum_\alpha\lambda_\alpha(x)\nabla_{g_x}u_\alpha(m)\right|_{g_x}^2  \\
 &\leq
 \sum_\alpha\lambda_\alpha(x)|\nabla_{g_x}u_\alpha(m)|_{g_x}^2  \notag\\
 &\leq
 \eta\sum_\alpha\lambda_\alpha(x)\bigl(-\Delta_{g_x}u_\alpha(m)\bigr)  \notag\\
 &=
 \eta\bigl(-\Delta_{g_x}U_x(m)\bigr). \notag
\end{align}

If \(x\notin U_0^+\), then \(\sum_\alpha\lambda_\alpha(x)=1\). Therefore, using
\eqref{eq:u-alpha-estimates},
\begin{equation}\label{eq:sum-laplace}
 -\Delta_{g_x}U_x(m)
 =
 \sum_\alpha\lambda_\alpha(x)\bigl(-\Delta_{g_x}u_\alpha(m)\bigr)
 \geq L.
\end{equation}
Using the formula for the conformal change of the scalar curvature \eqref{eq:conformal-scalar}, \eqref{eq:sum-gradient}, and
\eqref{eq:sum-laplace}, we get
\begin{align*}
 e^{2U_x(m)}\scal_{G_x}(m)
 &\geq
 \rho_0
 +2(n-1)\bigl(-\Delta_{g_x}U_x(m)\bigr)
 -(n-1)(n-2)|\nabla_{g_x}U_x(m)|_{g_x}^2  \\
 &\geq
 \rho_0
 +(n-1)\bigl(2-(n-2)\eta\bigr)
 \bigl(-\Delta_{g_x}U_x(m)\bigr) \\
 &=
 \rho_0+(n-1)a\bigl(-\Delta_{g_x}U_x(m)\bigr) \\
 &\geq
 \rho_0+(n-1)aL>0.
\end{align*}

It remains to consider the case \(x\in U_0^+\). By \eqref{eq:positive-near-X0}, the original scalar
curvature satisfies \(\scal_{g_x}(m)>0\) at all points \((x,m)\in C\) with \(x\in U_0^+\). Also, by
\eqref{eq:sum-gradient},
\[
 2\bigl(-\Delta_{g_x}U_x(m)\bigr)
 -(n-2)|\nabla_{g_x}U_x(m)|_{g_x}^2
 \geq
 \bigl(2-(n-2)\eta\bigr)\bigl(-\Delta_{g_x}U_x(m)\bigr)
 =
 a\bigl(-\Delta_{g_x}U_x(m)\bigr)
 \geq 0.
\]
Thus, the conformal deformation preserves positivity of the scalar curvature at
$(x,m)\in C$ with $x\in U_0^+$. Hence, $\scal_{G_x}(m)>0$ for every $(x,m)\in C$ with $x\in U_0^+$ as well.

Therefore, $\scal_{G_x}>0$ at every point of $C$. By compactness of $C$ and continuity, there is an
open neighborhood $\Omega\subset X\times M$ of $C$ such that $\scal_{G_x}>0$ on
$\Omega_x$ for every $x\in X$, as desired. 
\end{proof}

\subsubsection{Bertelson's result}\label{sec:parametric-input}

Let us now prove \cref{cor:finite-critical-path}, which is based on a theorem of Bertelson.
Recall that a set $A\subset B$ is called \emph{residual} if it is the countable intersection of open and dense subsets.
If $X$ is a foliated manifold and $L_x$ is the leaf through $x\in X$, then $x$ is a \emph{leafwise critical point} of $h\in C^\infty(X)$ if it is a critical point of the restriction $h|_{L_x}$.

\begin{proposition}[{\cite[Proposition~3.21]{Bertelson2002}}]
\label{thm:bertelson-isolated}
Let $X$ be a closed manifold  equipped with a smooth foliation. There is a residual and hence dense subset
of $C^\infty(X)$, with respect to the $C^\infty$-topology, consisting of
functions $h$ such that, for every leaf $L$ of the foliation, all the critical
points of $h|_L$ are isolated in $L$.
\end{proposition}
We briefly recall the ingredients entering the
result. The
relevant argument occupies only \cite[pp.~386--391]{Bertelson2002}. Apart from
the basic definitions, its main external inputs are Thom's (jet) transversality theorem
\cite[Theorem~3.1]{Bertelson2002}, Mather's finite determinacy theorem
\cite[Theorem~3.15]{Bertelson2002}, and two classical theorems of Whitney on
real algebraic sets \cite[Theorems~3.17 and~3.18]{Bertelson2002}. Once these
classical results are taken for granted, the remaining argument is short and can be
checked directly. 

\begin{corollary}\label{cor:igusa-finite-critical}
  Let $B$ and $N$ be compact and smooth manifolds, possibly with boundary, and let $F:B\times N\to\mathbb R$ be smooth. 
  Then $F$ can be arbitrarily closely approximated in the $C^\infty$-topology by a smooth map $\widetilde F:B\times N\to\mathbb R$ such that every fiber $\widetilde F_b:N\to\mathbb R$ has finitely many critical points.
\end{corollary}

\begin{proof}
Let us first assume that $\partial B=\partial N=\varnothing$. Endow $B\times N$ with the product foliation whose leaves are the fibers $\{b\}\times N$, $b\in B$. By 
\cref{thm:bertelson-isolated}, choose $\widetilde F$ in a prescribed
$C^\infty$-neighborhood of $F$ such that the critical points of its restriction to
every leaf are isolated. These leafwise critical points are precisely the
critical points of the functions $\widetilde F_b:N\to\mathbb R$, so every critical point of
every $\widetilde F_b$ is isolated. For a fixed $b$, the critical set
\[
 \Crit(\widetilde F_b)=(\mathrm{d}\widetilde F_b)^{-1}(0)
\]
is a closed subset of the compact manifold $N$, hence compact. A compact set
whose points are all isolated is finite, thus concluding the proof.

If either $B$ or $N$ has boundary, extend $F$ to the product of suitable closed
doubles, apply the preceding argument there, and restrict the resulting
function to $B\times N$.
\end{proof}

The following is precisely the result we shall use.

\begin{proof}[Proof of \cref{cor:finite-critical-path}]
Apply \cref{cor:igusa-finite-critical} to $B=X\times [0,1]$, $N=M$, and $F((x,t),m)=H_{x,t}(m)$.
\end{proof}

\section{Proof of \texorpdfstring{\cref{main:existence-with-invariance-and-boundary}}{Theorem B}}\label{sec:THMB}

The key idea is to transfer the problem from $\Riem^\gamma(M)$ to $\Q_{c_{n,\gamma}}(M)$, through the identity in \eqref{eqn:KeyRgammaMtoQcM} and Barta's trick as stated in \cref{cor:Barta}. 
Hence, to prove \cref{main:existence-with-invariance-and-boundary} it is enough to construct a $\Gamma$-invariant metric $h$ and a $\Gamma$-invariant function $\varphi$ such that $\scal_h+c_{n,\gamma}|\nabla_h \varphi|^2_h>0$, satisfying $\varphi\equiv 0$ in a neighborhood of $\partial M$, $h|_{T\partial M}=q$, and $\mathrm{II}_h=k$.
In the case of closed manifolds, this will be done in \cref{thm:mainnew2}, where we exploit the theory developed in \cite{BaerHankeLocalFlexibility}.
In \cref{thm:mainnew} we prove the analogous statement for the case $\partial M\not=\varnothing$, again building on the main results of \cite{BaerHankeLocalFlexibility} and the $\Gamma$-invariant analog of \cite[Remark 3.2]{BaerHankeBoundaryConditions}.

\begin{proposition}\label{thm:mainnew2}
Let $\Gamma$ be a compact Lie group, let $n\geq 2$, and let $M^n$ be a
closed connected smooth $\Gamma$-manifold. Assume that the $\Gamma$-action
is not transitive. Then, for every $c>0$, there exist a smooth
$\Gamma$-invariant Riemannian metric $h$ on $M$ and a smooth
$\Gamma$-invariant function $\varphi$ on $M$ such that
\begin{equation}\label{eq:strict-main2}
    \scal_h+c|\nabla_h\varphi|_h^2>0
    \qquad\text{on }M.
\end{equation}
\end{proposition}

\begin{proof}
By equivariant Morse theory, there exists a smooth $\Gamma$-invariant
Morse function 
\[
    f:M\longrightarrow\mathbb R
\]
in the sense of \cite{WassermanEquivariantDifferentialTopology} and \cite{MayerInvariantMorseFunctions}, see also the summary in
\cite[Section 2, pp.~3--4]{BaerHankeLocalFlexibility}.
The equivariant Morse lemma
\cite[Satz~1.3]{MayerInvariantMorseFunctions} and compactness imply that
the critical set is a finite disjoint union of critical orbits
\[
    \Crit(f)=\mathcal O_1\sqcup\cdots\sqcup\mathcal O_N.
\]

Put $
    d_i\coloneqq \codim_M\mathcal O_i$. 
Then $d_i\geq 1$ for every $i$. Indeed, if $d_i=0$, the embedded orbit
$\mathcal O_i$ is open in $M$. It is also closed because $\Gamma$ is
compact; hence, connectedness would give $\mathcal O_i=M$, contrary to
the nontransitivity assumption.
\smallskip

Choose a $\Gamma$-invariant Riemannian metric $h_0$ on $M$. Let
\[
    r_i(x)\coloneqq \dist_{h_0}(x,\mathcal O_i).
\]
After
choosing sufficiently small pairwise disjoint $\Gamma$-invariant tubular neighborhoods
$U_i$ of the orbits $\mathcal{O}_i$, $r_i^2$ is smooth on $U_i$. Choose
$\Gamma$-invariant functions $\chi_i\in C^\infty_c(U_i,[0,1])$ that
are $\equiv 1$ near $\mathcal O_i$, and, for $\Lambda>0$, define
\[
    \psi\coloneqq -\Lambda\sum_{i=1}^N\chi_i r_i^2,
    \qquad
    h\coloneqq e^{2\psi}h_0,
\]
where $\chi_i r_i^2$ is extended by zero outside $U_i$. This is the
single-metric version of the conformal construction in
\cite[proof of Lemma~2.5, (2.6)]
{BaerHankeLocalFlexibility}. Along $\mathcal O_i$ one has
\[
    \psi=0,\qquad
    \dd\psi=0,\qquad
    \Delta_{h_0}\psi=-2\Lambda d_i.
\]
The conformal scalar-curvature formula \eqref{eq:conformal-scalar} therefore gives
\begin{equation}\label{eq:scalar-critical-orbit}
    \scal_h\big|_{\mathcal O_i}
    =
    \scal_{h_0}\big|_{\mathcal O_i}
    +4\Lambda(n-1)d_i.
\end{equation}
Since there are only finitely many compact critical orbits and $d_i\geq1$,
we may choose $\Lambda$ so large that $\scal_h>2$ on $\Crit(f)$.
Consequently, $\scal_h>1$ on some neighborhood of $\Crit(f)$.

Let
\[
    Z\coloneqq \{x\in M\mid\scal_h(x)\leq0\}.
\]
If $Z=\varnothing$, we take $\varphi=0$. Otherwise, $Z$ is compact and
disjoint from $\Crit(f)$, and hence
\[
    \mu\coloneqq \min_Z|\nabla_hf|_h^2>0,
    \qquad
    B\coloneqq \max_Z(-\scal_h)<\infty.
\]
Choose $A>0$ such that $cA^2\mu>B$ and set $\varphi\coloneqq Af$. Then, on $Z$,
\[
    \scal_h+c|\nabla_h\varphi|_h^2
    \geq -B+cA^2\mu>0,
\]
whereas on $M\setminus Z$ the scalar curvature is already positive.
Thus, \eqref{eq:strict-main2} holds. Finally, note that both $h$ and $\varphi$ are
$\Gamma$-invariant, as desired.
\end{proof}

Turning to the case of non-empty boundary, we begin with the following lemma, which is the $\Gamma$-invariant analog of \cite[Remark 3.2]{BaerHankeBoundaryConditions}.

\begin{lemma}
\label{lem:boundary-collar-detailed}
Let $n\geq 2$, and let $\Gamma$ be a compact Lie group. Let $M^n$ be a compact smooth $\Gamma$-manifold with nonempty boundary.
Let $q$ be a $\Gamma$-invariant Riemannian metric on $\partial M$, let
$k$ be a $\Gamma$-invariant symmetric $(0,2)$-tensor on $\partial M$, and let $\sigma\geq 0$.
Then there is an $\varepsilon > 0$, an open $\Gamma$-invariant collar neighborhood $U \approx [0,\varepsilon) \times \partial M$ of $\partial M$ in $M$ and a metric $h_{\mathrm{col}}$ on
$U$ such that
\[
  h_{\mathrm{col}}|_{T\partial M}=q,
  \qquad
  \mathrm{II}_{h_{\mathrm{col}}}=k,
  \qquad
  \scal_{h_{\mathrm{col}}}>\sigma .
\]
Here, the second fundamental form is computed along $\partial M \subset U$ and with respect to the outward unit normal.
\end{lemma}
\begin{proof}
    The proof is a minor adaptation to the $\Gamma$-invariant case of the argument of \cite[Remark 3.2]{BaerHankeBoundaryConditions}, which covers the non-equivariant case. 
    
    First, let us fix an auxiliary $\Gamma$-invariant metric $g_0$ on $M$. 
    The normal exponential map of $g_0$ at the boundary gives an $\varepsilon > 0$ and a  $\Gamma$-equivariant diffeomorphism of $[0,\varepsilon)\times \partial M$ with some open neighborhood $U$ of $\partial M$. 
      In this identification, $\Gamma$ acts trivially on the $[0,\varepsilon)$-factor.
      Now, denoting the standard coordinate on $[0,\varepsilon)$ by $t$, consider the following metric: 
    \begin{equation}\label{eq:collar-metric}
  h_{\mathrm{col}}\coloneqq \dd t^2+q_t,
  \qquad \text{where} \qquad
  q_t\coloneqq q-2t k-\Lambda t^2 q.
\end{equation}
By construction, $h_{\mathrm{col}}$ is $\Gamma$-invariant. 
A short computation shows that $h_{\mathrm{col}}|_{T\partial M}=q$, and $\mathrm{II}_{h_{\mathrm{col}}}=k$. Moreover, first choosing $\Lambda$ big enough, and then $\varepsilon$ small enough, we also have $\scal_{h_{\mathrm{col}}}>\sigma$, compare \cite[Remark 3.2]{BaerHankeBoundaryConditions}.
\end{proof}

\begin{proposition}\label{thm:mainnew}
Let $\Gamma$ be a compact Lie group, let $n\geq 2$, and let $M^n$ be a
compact connected smooth $\Gamma$-manifold with nonempty boundary.  
Let $c>0$. Let $q$
be any $\Gamma$-invariant boundary metric and let $k$
be any $\Gamma$-invariant symmetric $(0,2)$-tensor on $\partial M$. 

Then there are a smooth
$\Gamma$-invariant Riemannian metric $h$ on $M$ and a smooth
$\Gamma$-invariant function $\varphi$ on $M$ such that
\begin{equation}\label{eq:strict-main}
  \scal_h+c|\nabla_h\varphi|_h^2>0\qquad\text{on }M,
\end{equation}
\begin{equation}\label{eq:boundary-data}
  h|_{T\partial M}=q,\qquad \mathrm{II}_h=k, \qquad 
 \varphi\equiv 0\,\quad \text{on a neighborhood of $\partial M$}.
\end{equation}
\end{proposition}

\begin{proof}
By \cite[Theorem I]{BaerHankeLocalFlexibility}, there is a smooth $\Gamma$-invariant metric $g_+$ on $M$ with
\[
  \scal_{g_+}>0\quad\text{on }M.
\]
By \cref{lem:boundary-collar-detailed}, we find an invariant metric
$h_{\mathrm{col}}$ on a collar $U \approx [0,\varepsilon)\times \partial M$ of $\partial M$ in $M$ which induces $q$ on the boundary, has
second fundamental form $k$ with respect to the outward unit normal, and has positive scalar curvature on $[0,a]\times \partial M $, with $a<\varepsilon$.

Fix $b$ such that $0<a<b<\varepsilon$.
Let
$\rho:[0,\varepsilon)\to[0,1]$ be smooth, with $\rho=1$ in a neighborhood of $[0,a]$
and $\rho=0$ in a neighborhood of $[b,\varepsilon)$.  On the collar define
\begin{equation}\label{eq:boundary-interpolation}
  h=\rho(t)h_{\mathrm{col}}+(1-\rho(t))g_+,
\end{equation}
and set $h=g_+$ outside the collar. We have that $h$ is a
smooth $\Gamma$-invariant metric. Since it agrees with $h_{\mathrm{col}}$ near the
boundary, we have
\[
  h|_{T\partial M}=q,
  \qquad \mathrm{II}_h=k.
\]
Moreover, $\scal_h>0$ both in $[0,a]\times\partial M$ and outside $[0,b]\times\partial M$.  Thus,
\[
  Z_h\coloneqq \{\scal_h\leq0\}\Subset (a,b)\times\partial M.
\]
If $Z_h=\varnothing$, choose $\varphi=0$, and the proof is complete. Otherwise, there are numbers $
  0<a<a_0<a_1<b$
such that the $t$-coordinate of every point of $Z_h$ lies in
$(a_0,a_1)$.  Choose a smooth function
$F:[0,\varepsilon)\to\mathbb R$ satisfying
\[
  F\equiv 0\text{ in $[0,a]$},
  \qquad F'(t)>0\text{ for }t\in[a_0,a_1],
  \qquad F\text{ $\equiv\Lambda$ on }[b,\varepsilon),
\]
for some $\Lambda>0$.
Define $f\coloneqq F(t)$ on the collar and extend it to be $\equiv \Lambda$ to
the rest of $M$.  The collar coordinate $t$ is invariant, hence $f$ is
invariant.  By construction, we have
  $f=\dd f=0$ near $\partial M$, while
$
  \dd f=F'(t)\dd t\neq0
$
on $Z_h$.  Hence, if $A\geq 0$ is large enough, and
$\varphi\coloneqq Af$, then
\[
  \scal_h+c|\nabla_h\varphi|_h^2>0
\]
on $Z_h$, and thus on all of $M$, as desired. Note that all conditions in \eqref{eq:boundary-data} hold as well.
\end{proof}

\begin{proof}[Proof of \cref{main:existence-with-invariance-and-boundary}]
    Let $c\coloneqq c_{n,\gamma}$ defined as in \eqref{eqn:cngamma}. Let $(h,\varphi)$ be the pair given by \cref{thm:mainnew} (if $\partial M\not=\varnothing$) or by \cref{thm:mainnew2} (if $\partial M=\varnothing$). Let $a\coloneqq \frac{2(n-1)}{\gamma}$. By \eqref{eqn:KeyRgammaMtoQcM} we have that, setting $g\coloneqq e^{2\varphi}h$, and using \eqref{eq:strict-main} (resp., \eqref{eq:strict-main2}),
    \begin{equation}\label{eqn:BartaBoundary}
    -\gamma\Delta_g(e^{-a\varphi})+\scal_g (e^{-a\varphi}) >0.
    \end{equation}
    Moreover, when $\partial M\neq\varnothing$, by \eqref{eq:boundary-data} we have that $g|_{T\partial M}=q$ and $\mathrm{II}_g=k$; and by \eqref{eqn:BartaBoundary} and the fact that $e^{-a\varphi}\equiv 1$ in a neighborhood of $\partial M$, we have $g(\nabla_g e^{-a\varphi},\nu_g)=0$. Hence, we conclude by \cref{cor:Barta} that
    \[
    -\gamma\Delta_g+\scal_g>0,
    \]
    as desired.
\end{proof}

\appendix
\crefalias{section}{appendix}

\section{Barta's trick}
We recall the elementary form of Barta's trick from \cite{Barta1937} that we will need in the paper. We claim no originality; we record these results for the reader's
convenience.
\begin{lemma}\label{lem:Barta}
    Let $(M,g)$ be a compact Riemannian manifold, possibly with nonempty boundary, and let  $\gamma\in\mathbb R$. Let $V\in C^\infty(M)$, and $L\coloneqq -\gamma\Delta+V$. Let $C^\infty(M)\ni u>0$. Then for every
$v\in C^\infty(M)$,
\begin{align}\label{eq:ground-state-transform-boundary}
\int_M\left(\gamma|\nabla v|^2+V v^2\right)\dvol
&=
\int_M\gamma u^2
\left|\nabla\left(\frac{v}{u}\right)\right|^2\dvol
+\int_M\frac{L u}{u}v^2\dvol \\
&\quad
+\gamma\int_{\partial M}
\frac{g(\nabla u,\nu)}{u}v^2\dd \sigma,
\end{align}
where $\nu$ is the outward unit normal to $\partial M$, and $\sigma$ is the surface measure on $\partial M$.
\end{lemma}
\begin{proof}
    It is enough to integrate the pointwise identity
    \begin{equation}\label{eq:pointwise-ground-state-transform}
\gamma|\nabla v|^2+V v^2
=
\gamma u^2\left|\nabla\left(\frac{v}{u}\right)\right|^2
+\frac{L u}{u}v^2
+\gamma\operatorname{div}\left(\frac{v^2}{u}\nabla u\right).
\end{equation}
The computation is classical.
\end{proof}
\begin{proposition}\label{cor:Barta}
    Let $(M,g)$ be a compact connected Riemannian manifold, possibly with nonempty boundary and unit outward normal $\nu$, and let  $\gamma>0$. Let $V\in C^\infty(M)$, $L\coloneqq -\gamma\Delta+V$, and let $L_N$ be the Neumann realization\footnote{Namely, the functions in the domain of $L_N$ have normal derivative $\equiv 0$ on the boundary.} of $L$. The following are equivalent:
    \begin{enumerate}
        \item $L>0$ in the sense of \eqref{eqn:-gammaDelta+r}.
        \item There is $C^\infty(M)\ni u>0$ such that $Lu>0$ on $M$ and $g(\nabla u,\nu)\geq 0$ on $\partial M$.
        \item There is $C^\infty(M)\ni u>0$ such that $Lu>0$ on $M$ and $g(\nabla u,\nu)= 0$ on $\partial M$.
    \end{enumerate}
    Moreover, if $\lambda_1^N(L)$ denotes the first Neumann eigenvalue
    of $L$, then
    \[
        \ker\bigl(L_N-\lambda_1^N(L)\bigr)=\mathbb R u_1,
    \]
    where $u_1>0$ is a first Neumann eigenfunction. In particular,
    whenever the equivalent conditions above hold, we have
    $\lambda_1^N(L)>0$.
\end{proposition}
\begin{proof}
    The implication $(3)\Rightarrow (2)$ is immediate. The implication $(2)\Rightarrow (1)$ readily follows from \cref{lem:Barta}. The implication $(1)\Rightarrow (3)$ follows by taking any $u>0$ in the first Neumann eigenspace of $L$ (see below to see why $u$ can be chosen $>0$).
    
    The second part of the statement is also classical. The fact that a first Neumann eigenfunction $u_1$ can be chosen positive is a consequence of three facts: if $v$ is a minimizer of the Rayleigh quotient, so is $|v|$; the strong maximum principle; and the Hopf lemma. Finally, let $u_1>0$ be a positive first Neumann eigenfunction and $v$ be any other first Neumann eigenfunction. It can be directly verified that 
    \[
    \mathrm{div}\left(u_1^2\nabla\left(\frac{v}{u_1}\right)\right)=0, \qquad g\left(\nabla\left(\frac{v}{u_1}\right),\nu\right)=0.
    \]
    Thus, integrating by parts $\frac{v}{u_1}\mathrm{div}\left(u_1^2\nabla\left(\frac{v}{u_1}\right)\right)$ gives $\int_M u_1^2\left|\nabla\left(\frac{v}{u_1}\right)\right|^2=0$, from which $v=cu_1$ for some $c\in \mathbb R$, as desired.
\end{proof}

\section{Continuous dependence of the first eigenfunction on parameters}
\label{appendix:cont_dep}
    We give a short justification for \cref{prop:first-eigenfunction-cts-dependence}.
    The argument is classical and has already appeared in the literature, e.g., in \cite[p. 155]{DaiWangWei2005} and \cite[Remark 2.2 \& Footnote (10)]{LiMantoulidis2023}. 
Fix $g_0\in\Riem(M)$ and write
\[
 u_0\coloneqq u_{g_0},
 \qquad
 \lambda_0\coloneqq \lambda_1(L_{g_0}).
\]
Fix a background metric to define Sobolev spaces (see, e.g., \cite[Chapter 4]{TaylorPDEI}). For an integer $s>n/2$,
let $\Riem_{s+2}(M)$ denote the space of $H^{s+2}$-Riemannian metrics on $M$ and
consider
\begin{equation}\label{eq:eigenpair-IFT-map}
 \begin{split}
 \mathcal F_s:\Riem_{s+2}(M)\times H^{s+2}(M)\times\mathbb R
 &\longrightarrow H^s(M)\times\mathbb R,\\
 \mathcal F_s(g,v,\lambda)
 &\coloneqq \left(L_gv-\lambda v,\ \int_Mv^2\,\dvol_g-1\right).
 \end{split}
\end{equation}
This is a smooth map between Banach manifolds. Its derivative in the
$(v,\lambda)$ variables at $(g_0,u_0,\lambda_0)$ is
\begin{equation}\label{eq:eigenpair-linearization}
 (w,\eta)\longmapsto
 \left((L_{g_0}-\lambda_0)w-\eta u_0,
 2\langle u_0,w\rangle_{L^2(g_0)}\right),
\end{equation}
seen as a map from $H^{s+2}(M)\times \mathbb R$ to $H^{s}(M)\times \mathbb R$. 
Since $\ker(L_{g_0}-\lambda_0)=\mathbb Ru_0$, elliptic Fredholm theory gives that \eqref{eq:eigenpair-linearization} is an isomorphism between the latter two spaces.

As a consequence, the Banach-space implicit function theorem gives smooth local maps 
\[
g\mapsto v_s(g)\in H^{s+2}(M), \qquad g\mapsto\lambda_s(g)\in\mathbb R,
\]
defined on a neighborhood $U_s$ of $g_0$ in $\Riem_{s+2}(M)$ 
such that 
\[
L_{g}v_s(g)=\lambda_s(g)v_s(g), \qquad \int_M v_s(g)^2\dvol_g=1.
\]
Since $s>n/2$, we have
$H^{s+2}(M)\hookrightarrow C^0(M)$. Moreover, considering that $u_0>0$, by possibly shrinking $U_s$ we can assume $v_s(g)>0$ for every $g\in U_s$.

We claim that for every $g\in U_s\cap \Riem(M)$, $v_s(g)$ must be the positive normalized first eigenfunction of $L_g$.
Let $\mu_s(g)\in\R$ be the first eigenvalue of $L_g$, and $w_s(g)\in \ker(L_{g}-\mu_s(g))$ be the unique (compare \cref{cor:Barta}) smooth positive function such that $\int_M w_s^2(g)\dvol_g=1$.
By self-adjointness of $L_g$ we have 
\[
\begin{aligned}
\mu_s(g) \int_M v_s(g)w_s(g)\dvol_g&= \int_M L_g w_s(g)v_s(g)\dvol_g = \int_M w_s(g)L_gv_s(g)\dvol_g \\
&= \lambda_s(g) \int_M v_s(g)w_s(g)\dvol_g.
\end{aligned}
\]
Thus, since $\int_M v_s(g)w_s(g)\dvol_g>0$, we get $\mu_s(g)=\lambda_s(g)$. Hence, $v_s\in \ker(L_g-\mu_s(g))\ni w_s$, which is a one-dimensional subspace by \cref{cor:Barta}. Since $\int v_s(g)^2\dvol_g=\int w_s(g)^2\dvol_g$, we conclude $w_s(g)=v_s(g)$, as desired.

Thus, for every $g\in U_s\cap\Riem(M)$ we have $v_s(g)=u_g$. Given $r\geq0$, choose $s$ sufficiently large that
$H^{s+2}(M)\hookrightarrow C^r(M)$. The local branches obtained at different
Sobolev indices agree on their common domains on $\Riem(M)$ by the uniqueness of the positive
normalized first eigenfunction. The preceding argument gives
continuity near $g_0$ of $g\mapsto u_g$ in every $C^r$ topology. Hence, the map is continuous near $g_0$ in the $C^\infty$-topology. Since $g_0$ is arbitrary, the map $g\mapsto u_g$ is continuous in the $C^\infty$-topology, as desired.\hfill$\Box$

% \bibliographystyle{plain} % or alpha, abbrv, unsrt
% \bibliography{ref}       % filename without .bib

\printbibliography

@article{Barta1937,
 author = {Barta, János},
 title = {Sur la vibration fondamentale d'une membrane.},
 fjournal = {Comptes Rendus Hebdomadaires des S{\'e}ances de l'Acad{\'e}mie des Sciences, Paris},
 journal = {C. R. Acad. Sci., Paris},
 issn = {0001-4036},
 volume = {204},
 pages = {472--473},
 year = {1937},
 language = {French},
 zbMATH = {2522248},
 JFM = {63.0762.02}
}

@article{BERW,
 author = {Botvinnik, Boris and Ebert, Johannes and Randal-Williams, Oscar},
 title = {Infinite loop spaces and positive scalar curvature},
 fjournal = {Inventiones Mathematicae},
 journal = {Invent. Math.},
 issn = {0020-9910},
 volume = {209},
 number = {3},
 pages = {749--835},
 year = {2017},
 language = {English},
 doi = {10.1007/s00222-017-0719-3},
 zbMATH = {6786971},
 Zbl = {1377.53067}
}

@article{CS,
 author = {Crowley, Diarmuid and Schick, Thomas and Steimle, Wolfgang},
 title = {Harmonic spinors and metrics of positive curvature via the {Gromoll} filtration and {Toda} brackets},
 fjournal = {Journal of Topology},
 journal = {J. Topol.},
 issn = {1753-8416},
 volume = {11},
 number = {4},
 pages = {1077--1099},
 year = {2018},
 language = {English},
 doi = {10.1112/topo.12081},
 zbMATH = {7015305},
 Zbl = {1466.57013}
}

@article{ERW,
 author = {Ebert, Johannes and Randal-Williams, Oscar},
 title = {The positive scalar curvature cobordism category},
 fjournal = {Duke Mathematical Journal},
 journal = {Duke Math. J.},
 issn = {0012-7094},
 volume = {171},
 number = {11},
 pages = {2275--2406},
 year = {2022},
 language = {English},
 doi = {10.1215/00127094-2022-0023},
 zbMATH = {7574071},
 Zbl = {1502.53083}
}

@article{HSS14,
 author = {Hanke, Bernhard and Schick, Thomas and Steimle, Wolfgang},
 title = {The space of metrics of positive scalar curvature},
 fjournal = {Publications Math{\'e}matiques},
 journal = {Publ. Math., Inst. Hautes {\'E}tud. Sci.},
 issn = {0073-8301},
 volume = {120},
 pages = {335--367},
 year = {2014},
 language = {English},
 doi = {10.1007/s10240-014-0062-9},
 zbMATH = {6381123},
 Zbl = {1321.58008}
}

@article{AkutagawaBotvinnik-relative-yamabe-invariant,
 author = {Akutagawa, Kazuo and Botvinnik, Boris},
 title = {The relative {Yamabe} invariant},
 fjournal = {Communications in Analysis and Geometry},
 journal = {Commun. Anal. Geom.},
 issn = {1019-8385},
 volume = {10},
 number = {5},
 pages = {935--969},
 year = {2002},
 language = {English},
 doi = {10.4310/CAG.2002.v10.n5.a2},
 zbMATH = {1925893},
 Zbl = {1034.58007}
}

@misc{PerelmanEntropy,
  title={The entropy formula for the Ricci flow and its geometric applications}, 
  author={Grisha Perelman},
  year={2002},
  eprint={math/0211159},
  archivePrefix={arXiv},
  url={https://arxiv.org/abs/math/0211159}, 
}

@article{HirschKazarasKhuriZhang,
 author = {Hirsch, Sven and Kazaras, Demetre and Khuri, Marcus and Zhang, Yiyue},
 title = {Spectral torical band inequalities and generalizations of the {Schoen}-{Yau} black hole existence theorem},
 fjournal = {IMRN. International Mathematics Research Notices},
 journal = {Int. Math. Res. Not.},
 issn = {1073-7928},
 volume = {2024},
 number = {4},
 pages = {3139--3175},
 year = {2024},
 language = {English},
 doi = {10.1093/imrn/rnad129},
 zbMATH = {7930694},
 Zbl = {1558.83060}
}

@article{KazdanWarner,
 author = {Kazdan, Jerry L. and Warner, Frank W.},
 title = {Scalar curvature and conformal deformation of {Riemannian} structure},
 fjournal = {Journal of Differential Geometry},
 journal = {J. Differ. Geom.},
 issn = {0022-040X},
 volume = {10},
 pages = {113--134},
 year = {1975},
 language = {English},
 doi = {10.4310/jdg/1214432678},
 zbMATH = {3464460},
 Zbl = {0296.53037}
}

@article{BotRos,
 author = {Botvinnik, Boris and Rosenberg, Jonathan},
 title = {Generalized positive scalar curvature on {{\(\text{spin}^c\)}} manifolds},
 fjournal = {Annals of Global Analysis and Geometry},
 journal = {Ann. Global Anal. Geom.},
 issn = {0232-704X},
 volume = {66},
 number = {4},
 pages = {20},
 note = {Id/No 18},
 year = {2024},
 language = {English},
 doi = {10.1007/s10455-024-09977-6},
 zbMATH = {7948989},
 Zbl = {1556.53051}
}

@article{DaiWangWei2005,
 author = {Dai, Xianzhe and Wang, Xiaodong and Wei, Guofang},
 title = {On the stability of {Riemannian} manifold with parallel spinors},
 fjournal = {Inventiones Mathematicae},
 journal = {Invent. Math.},
 issn = {0020-9910},
 volume = {161},
 number = {1},
 pages = {151--176},
 year = {2005},
 language = {English},
 doi = {10.1007/s00222-004-0424-x},
 zbMATH = {2190622},
 Zbl = {1075.53042}
}

@book{TaylorPDEI,
 author = {Taylor, Michael E.},
 title = {Partial differential equations {I}. {Basic} theory},
 edition = {3rd corrected and expanded edition},
 fseries = {Applied Mathematical Sciences},
 series = {Appl. Math. Sci.},
 issn = {0066-5452},
 volume = {115},
 isbn = {978-3-031-33858-8; 978-3-031-33861-8; 978-3-031-33859-5},
 year = {2023},
 publisher = {Cham: Springer},
 language = {English},
 doi = {10.1007/978-3-031-33859-5},
 zbMATH = {7756363},
 Zbl = {1527.35002}
}

@article{Palais1966,
  author   = {Palais, Richard S.},
  title    = {Homotopy theory of infinite dimensional manifolds},
  journal  = {Topology},
  volume   = {5},
  number   = {1},
  pages    = {1--16},
  year     = {1966},
  month    = mar,
  issn     = {0040-9383},
  doi      = {10.1016/0040-9383(66)90002-4},
  mrnumber = {MR0189028},
  zbl      = {0138.18302}
}

@article{WassermanEquivariantDifferentialTopology,
 author = {Wasserman, Arthur G.},
 title = {Equivariant differential topology},
 fjournal = {Topology},
 journal = {Topology},
 issn = {0040-9383},
 volume = {8},
 pages = {127--150},
 year = {1969},
 language = {English},
 doi = {10.1016/0040-9383(69)90005-6},
 zbMATH = {3341712},
 Zbl = {0215.24702}
}

@article{MayerInvariantMorseFunctions,
 author = {Mayer, Karl Heinz},
 title = {G-invariante {Morse}-{Funktionen}. ({G}-invariant {Morse} functions)},
 fjournal = {Manuscripta Mathematica},
 journal = {Manuscr. Math.},
 issn = {0025-2611},
 volume = {63},
 number = {1},
 pages = {99--114},
 year = {1989},
 language = {German},
 doi = {10.1007/BF01173705},
 url = {https://eudml.org/doc/155371},
 zbMATH = {4100211},
 Zbl = {0672.58007}
}

@article{MantoulidisSchoen2015,
  author  = {Mantoulidis, Christos and Schoen, Richard},
  title   = {On the {B}artnik mass of apparent horizons},
  journal = {Classical and Quantum Gravity},
  volume  = {32},
  number  = {20},
  pages   = {205002},
  year    = {2015},
  doi     = {10.1088/0264-9381/32/20/205002},
  url     = {https://doi.org/10.1088/0264-9381/32/20/205002}
}

@misc{BaerHankeLocalFlexibility,
  title={Contractibility of spaces of positive scalar curvature metrics with symmetry}, 
  author={Christian B{\"a}r and Bernhard Hanke},
  year={2026},
  eprint={2503.16232},
  archivePrefix={arXiv},
  url={https://arxiv.org/abs/2503.16232}, 
}

@incollection{BaerHankeBoundaryConditions,
 author = {B{\"a}r, Christian and Hanke, Bernhard},
 title = {Boundary conditions for scalar curvature},
 booktitle = {Perspectives in scalar curvature. Vol. 2},
 isbn = {978-981-12-4999-0; 978-981-12-4935-8; 978-981-12-4937-2},
 pages = {325--377},
 year = {2023},
 publisher = {Singapore: World Scientific},
 language = {English},
 doi = {10.1142/9789811273230_0010},
 zbMATH = {7733276},
 Zbl = {1530.53054}
}

@misc{BamlerKleiner2019,
      title={Ricci flow and contractibility of spaces of metrics}, 
      author={Richard H. Bamler and Bruce Kleiner},
      year={2019},
      eprint={1909.08710},
      archivePrefix={arXiv},
      url={https://arxiv.org/abs/1909.08710}, 
}

@article{Bertelson2002,
 author = {Bertelson, Mélanie},
 title = {A {{\(h\)}}-principle for open relations invariant under foliated isotopies.},
 fjournal = {The Journal of Symplectic Geometry},
 journal = {J. Symplectic Geom.},
 issn = {1527-5256},
 volume = {1},
 number = {2},
 pages = {369--425},
 year = {2002},
 language = {English},
 doi = {10.4310/JSG.2001.v1.n2.a6},
 zbMATH = {1911435},
 Zbl = {1036.53017}
}

@article{LiMantoulidis2023,
 author = {Li, Chao and Mantoulidis, Christos},
 title = {Metrics with {{\(\lambda_1(-\Delta + k R) \ge 0\)}} and flexibility in the {Riemannian} {Penrose} inequality},
 fjournal = {Communications in Mathematical Physics},
 journal = {Commun. Math. Phys.},
 issn = {0010-3616},
 volume = {401},
 number = {2},
 pages = {1831--1877},
 year = {2023},
 language = {English},
 doi = {10.1007/s00220-023-04679-9},
 zbMATH = {7707363},
 Zbl = {1526.53035}
}

@incollection{GromovFourLectures,
 author = {Gromov, Misha},
 title = {Four lectures on scalar curvature},
 booktitle = {Perspectives in scalar curvature. In 2 volumes},
 isbn = {978-981-12-4935-8; 978-981-12-4937-2; 978-981-12-4998-3; 978-981-12-4999-0},
 pages = {1--514},
 year = {2023},
 publisher = {Singapore: World Scientific},
 language = {English},
 doi = {10.1142/9789811273223_0001},
 zbMATH = {7733259},
 Zbl = {1532.53003}
}

@article {Rub,
 author = {Ruberman, Daniel},
 title = {Positive scalar curvature, diffeomorphisms and the {Seiberg}-{Witten} invariants},
 fjournal = {Geometry \& Topology},
 journal = {Geom. Topol.},
 issn = {1465-3060},
 volume = {5},
 pages = {895--924},
 year = {2001},
 language = {English},
 doi = {10.2140/gt.2001.5.895},
 url = {https://eudml.org/doc/122301},
 zbMATH = {1760363},
 Zbl = {1002.57064}
}

\end{document}